\documentclass[10pt]{amsart}
\usepackage{times,amsmath,amsbsy,amssymb,amscd,mathrsfs}
\usepackage{slashbox}\usepackage{graphicx, caption, subcaption, epstopdf,wrapfig,chemarrow}

\usepackage{algorithm2e} 
\usepackage{multicol,multirow}
\usepackage{mathtools}
\usepackage[usenames,dvipsnames,svgnames,table]{xcolor}
\usepackage[all]{xy}
\usepackage{wrapfig}
\usepackage{tcolorbox}
\usepackage{stmaryrd}

\usepackage{tikz,tikz-cd}
\usepackage[utf8]{inputenc}
\usepackage{pgfplots} 
\usepackage{pgfgantt}
\usepackage{pdflscape}
\pgfplotsset{compat=newest} 
\pgfplotsset{plot coordinates/math parser=false}
\newlength\fwidth

\usepackage[numbered]{mcode}

\definecolor{myBlue}{rgb}{0.0,0.0,0.55}
\usepackage[pdftex,colorlinks=true,citecolor=myBlue,linkcolor=myBlue]{hyperref}

\usepackage[hyperpageref]{backref}

\usepackage{comment,enumerate,multicol,xspace}

  \newcounter{mnote}
  \let\oldmarginpar\marginpar
    \renewcommand\marginpar[1]{\-\oldmarginpar[\raggedleft\footnotesize #1]%
    {\raggedright\footnotesize #1}}

\newtheorem{theorem}{Theorem}[section]
\newtheorem{lemma}[theorem]{Lemma}
\newtheorem{corollary}[theorem]{Corollary}
\newtheorem{proposition}[theorem]{Proposition}

\newtheorem{remark}[theorem]{Remark}

\newcommand{\mc}{\mcode}

\newcommand{\vertiii}[1]{{\left\vert\kern-0.25ex\left\vert\kern-0.25ex\left\vert #1 
    \right\vert\kern-0.25ex\right\vert\kern-0.25ex\right\vert}}

\usepackage{slashbox}
\usepackage{booktabs}
\usepackage[foot]{amsaddr}

\newcommand{\IV}{\mathcal{I}_{\mathcal V}}
\newcommand{\IQ}{\mathcal{I}_{\mathcal Q}}
\newcommand{\IU}{T_{\mathcal U}}
\newcommand{\IP}{T_{\mathcal P}}
\newcommand{\mV}{\mathcal V}
\newcommand{\mQ}{\mathcal Q}
\newcommand{\mU}{\mathcal U}
\newcommand{\mP}{\mathcal P}

\newcommand{\correction}{\textcolor{black}}

\title[Transformed Primal-Dual Methods]{Transformed Primal-Dual Methods\\
 for Nonlinear Saddle Point Systems}
\author{Long Chen $^{1,\ast}$}
\address{$^{\ast}$Corresponding author}
\address{$^1$Department of Mathematics, University of California at Irvine, Irvine, CA 92697, USA}%
\email{chenlong@math.uci.edu}
\author{Jingrong Wei$^1$}
\email{jingronw@uci.edu}%

\begin{document}
\begin{abstract} 
A transformed primal-dual (TPD) flow is developed for a class of nonlinear smooth saddle point system. The flow for the dual variable contains a Schur complement which is strongly convex. Exponential stability of the saddle point is obtained by showing the strong Lyapunov property. Several TPD iterations are derived by implicit Euler, explicit Euler, implicit-explicit and \correction {Gauss-Seidel methods with accelerated overrelaxation} of the TPD flow. Generalized to the symmetric TPD iterations, linear convergence rate is preserved for convex-concave saddle point systems under assumptions that the regularized functions are strongly convex. The effectiveness of augmented Lagrangian methods can be explained as a regularization of the non-strongly convexity and a preconditioning for the Schur complement. The algorithm and convergence analysis depends crucially on appropriate inner products of the spaces for the primal variable and dual variable.  A clear convergence analysis with nonlinear inexact inner solvers is also developed. 
 \end{abstract}

\maketitle



\section*{Statements and Declarations}
The authors were supported by NSF DMS-1913080 and DMS-2012465.

\section*{Acknowledgement}
We would like to thank Dr. Jianchao Bai, Dr. Ruchi Guo, and Dr. Solmaz Kia for valuable suggestions, especially the discussion on the augmented Lagrangian methods. We also thank Dr. Hao Luo for careful proof reading \correction{and discussion on the Gauss-Seidel method with accelerated overrelaxation}.

\newpage


\section{Introduction}
\subsection{Problem setting}
Consider a class of nonlinear \correction{smooth} saddle point systems:
\begin{equation}\label{eq: min-max problem}
    \min_{u\in \mathbb{R}^m} \max_{p \in \mathbb{R}^n} \mathcal{L}(u,p) = f(u) - g(p) + (Bu,p)
\end{equation}
where $B$ is an $n\times m$ matrix, $n \leq m$, with full row rank, $f(u), g(p)$ are smooth convex functions with convexity constant $\mu_f, \mu_g$, and $\nabla f(u), \nabla g(p)$ are Lipschitz continuous with Lipschitz constants $L_f, L_g$, respectively. The point $(u^*, p^*)$ solves the min-max problem~\eqref{eq: min-max problem} is said to be a saddle point of $\mathcal{L}(u,p)$, that is
$$\mathcal{L}(u^*, p) \leq \mathcal{L}(u^*, p^*) \leq  \mathcal{L}(u, p^*)\quad \forall \ (u,p)\in \mathbb R^m \times \mathbb R^n.$$ 
Convex optimization problems with affine equality constraints can be rewritten into a saddle point system~\eqref{eq: min-max problem}: 
\begin{equation}\label{eq: one-block affine equality constrained optimization system}
    \begin{aligned}
    \min_{u\in \mathbb{R}^m} f(u)\\
    \text{subject to} \quad Bu = b.
    \end{aligned}
\end{equation}
Then $p$ is the Lagrange multiplier to impose the constraint $Bu = b$ and $\mathcal{L} (u,p) = f(u) - (b,p) \correction{+ (Bu, p)}$. Note that $\mu_g = 0$ since $g(p)= (b, p)$ is linear and not strongly convex.

\correction{The saddle point $(u^*, p^*)$ satisfies the first order necessary condition for the critical point of $\mathcal{L}(u, p)$:
\begin{equation}\label{eq:critical point system}
    \begin{aligned}
    \nabla f(u^*) + B^{\top}p = 0, \\
    Bu^* - \nabla g(p^*) = 0.
    \end{aligned}
\end{equation}
If $\nabla f(u)=Au$ and $\nabla g(p)=Cp$, where $A, C$ are symmetric positive semidefinite matrices, one can recover the linear saddle point system:
\begin{equation}
 \left(\begin{array}{cc}
A & B^{\top} \\
B & -C
\end{array}\right)\left(\begin{array}{l}
u^* \\
p^*
\end{array}\right)=\left(\begin{array}{l}
f \\
g
\end{array}\right),   
\end{equation}
which arises in computational fluid dynamics~\cite{boon2022robust}, mixed finite element approximation of PDEs~\cite{chen2018multigrid, chen2015some,  huang2018multigrid}, optimal control problems~\cite{zeng2022dynamical}, etc (see~\cite{benzi2005numerical} and references therein).
}

\correction{For solving~\eqref{eq:critical point system}, the Arrow--Hurwicz and Uzawa methods proposed
in~\cite{arrow1958studies} is one of the earliest and most fundamental method. The pioneer work inspired influential algorithms such as the extragradient algorithm~\cite{korpelevich1976extragradient}, the Popov's modified method~\cite{popov1980modification} (also known as optimistic gradient descent-ascent methods). For strongly convex-strongly concave systems, i.e., $\mu_f> 0$ and $\mu_g> 0$, linear convergence of the extragradient algorithm was established in~\cite{korpelevich1976extragradient}. For general convex-concave systems only sub-linear rates are achieved in~\cite{golowich2020last, mokhtari2020convergence,tran2020non,yoon2021accelerated}.}

One may ask a question immediately: can we retain linear convergence rate only with partially strong convexity, i.e., $\mu_f> 0$ but $\mu_g = 0$, which covers the most important \correction{constrained optimization problem}~\eqref{eq: one-block affine equality constrained optimization system}? The answer is yes. When $f$ is strongly convex, its convex conjugate exists, i.e., $f^*(\xi) = \max_{u \in \mathbb{R}^m}  \ (\xi, u)- f(u)$ is well defined and convex. Then~\eqref{eq: min-max problem} is equivalent to the composite optimization problem without constraints:
\begin{equation}\label{eq: dual problem}
    \begin{aligned}
    \min_{p\in \mathbb{R}^n} f^*(-B^\top p) + g(p).
    \end{aligned}
\end{equation}
Notice $f^*$ is strongly convex since $\nabla f$ is Lipschitz continuous \correction{and $B$ is full row rank},~\eqref{eq: dual problem} is a strongly convex optimization problem with respect to the dual variable $p$. If $f^*$ and $\nabla f^*$ is \correction{computationally} available, convex optimization methods can be applied to solve~\eqref{eq: dual problem} and obtain linear convergence with strong convexity of $f^*$. Inexact Uzawa methods (IUM) \correction{for linear saddle point systems~\cite{ bacuta2006unified, bacuta2009schur, bank1989class, bramble2000uzawa, cheng2000nonlinear, elman1994inexact, peters2005fast,  queck1989convergence} and nonlinear saddle point systems
\cite{chen2015some, chen2017solving, chen1998global,chen1998preconditioned, hu2006nonlinear}} can be thought of as an inexact evaluation of $\nabla f^*$ for solving~\eqref{eq: dual problem} and achieving linear convergence rate. Usually a nonlinear inner iteration terminated with a certain accuracy for computing $\nabla f^*$ is required \correction{\cite{bacuta2006unified,bacuta2009schur, chen1998global,cheng2000nonlinear, Hu.Q;Zou.J2002,hu2006nonlinear,peters2005fast, song2019inexact}}. 


%
%




\subsection{Flows}
\correction{We shall study the iterative methods from the ODE solvers point of view. Namely we treat $(u(t), p(t))$ as continuous functions of $t$ and design ODE systems so that the saddle point $(u^*, p^*)$ is an equilibrium point of the corresponding dynamic system. Then we apply ODE solvers to obtain various iterative methods. By doing this way, we can borrow the analysis tools for dynamic systems to prove the stability and convergence theory of ODE solvers.}

The main stream \correction{in this direction} is the primal-dual gradient dynamics, which treat $u$ as the primal variable and $p$ as the dual variable and follows the primal-dual (PD) flow~\cite{arrow1958studies}: 
\begin{equation}\label{eq: primal-dual flow}
\left\{\begin{aligned}
u' &= -\partial_u \mathcal{L}(u,p) = -\nabla f(u) - B^\top p \\
p' &= \partial_p \mathcal{L}(u,p)  = Bu -\nabla g(p) 
\end{aligned}\right. ,
\end{equation} 
where $u', p'$ are taking the derivative of $t$. \correction{The exponential stability of the equilibrium point $(u^*, p^*)$ is shown in~\cite{qu2018exponential} for problem \eqref{eq: one-block affine equality constrained optimization system} and asymptotic convergence for general convex-concave systems can be found in~\cite{cherukuri2017saddle} and references therein. Then ODE solvers for \eqref{eq: primal-dual flow} will lead to several iterative methods and the linear convergence may be obtained using the exponential stability in the continuous level.}


For linear saddle point problems, we have the following factorization:
\begin{equation}\label{eq: block factorization}
\begin{pmatrix}
A & B^{\top} \\
B & -C
\end{pmatrix}=\begin{pmatrix}
I & 0 \\
BA^{-1} & I
\end{pmatrix}\begin{pmatrix}
A & 0 \\
0 & -S
\end{pmatrix}\begin{pmatrix}
I & A^{-1} B^{\top} \\
0 & I
\end{pmatrix},
\end{equation}
where $A\in \mathbb{R}^{m\times m}$ is symmetric positive definite (SPD), $B\in \correction{\mathbb{R}^{n\times m}}$ is surjective, $C\in \mathbb{R}^{n\times n}$ is symmetric and semi-positive definite, and $S=B A^{-1} B^{\top}+C$ is the Schur complement of $A$. 
The triangular matrix in~\eqref{eq: block factorization}
can be viewed as a change of coordinate. By changing to the correct `coordinate', the primal and dual variables are decoupled and the Schur complement $S$ defines a strongly convex function of the dual variable; see~\eqref{eq: dual problem}. 

Generalized to nonlinear systems, we consider a change of variable $ v = u + \IV^{-1} B^\top p$ where $\IV$ is an SPD matrix. 
Based on this transformation, we propose the following transformed primal-dual (TPD) flow
\begin{equation}\label{eq: TPD}
\left \{ \begin{aligned}
u' &=   -\IV^{-1}\partial_u \mathcal{L}(u,p) = \correction{-\IV^{-1} (\nabla f(u) + B^\top p )}\\
p' &= \IQ^{-1}\left ( \partial_p \mathcal{L}(u,p)-B\IV^{-1}\partial_u \mathcal{L}(u,p) \right ) = \correction{- \IQ^{-1}\left [  \nabla g_B(p) - B u +B\IV^{-1} \nabla f(u)   \right ]},
\end{aligned} \right.
\end{equation}
where $\IQ$ is another SPD matrix and $g_B(p):=g(p) + \frac{1}{2}\correction{p^{\top}}B\IV^{-1}B^{\top}p$. 
Here following~\cite{Bramble;Pasciak;Vassilev:1997Analysis} and~\cite{Zulehner2011}, the TPD flow is posed in appropriate inner products induced by SPD matrices $\IV$ and $\IQ$ on $\mathbb{R}^m$ and $\mathbb{R}^n$, respectively. After the transformation, \correction{the gradient of} the Schur complement $B\IV^{-1}B^{\top} p$ is added to $\nabla g(p)$. Even $\mu_g = 0$, the function $g_B$  is strongly convex and thus the exponential stability for the TPD flow can be established. More precisely, 
if $(u(t), p(t))$ solves the TPD flow~\eqref{eq: TPD}, we shall prove the exponential decay 
\begin{equation}\label{eq: Intro exponential decay}
\mathcal{E}(u(t),p(t)) \leq e^{-\mu t}\mathcal{E}(u(0),p(0)), \quad t>0,
\end{equation}
where the Lyapunov function 
\begin{equation}\label{eq: Intro Lyapunov}
\mathcal{E}(u,p) = \frac{1}{2} \|u-u^*\|^2_{\mathcal{I}_\mV} + \frac{1}{2} \|p-p^*\|^2_{\mathcal{I}_\mQ},
\end{equation}
\correction{and $\mu = \min \{\mu_{f,\IV}, (2-L_{f, \IV})\mu_{g_{B}, \IQ})\}$ with assumption $L_{f, \IV}<2$ which can be satisfied by rescaling}.

In Fig. \ref{fig:PDvsTPD}, we present numerical results for the example $\mathcal L(u,p) = \frac{1}{2}u^2 - up$ with $u,p\in \mathbb R$. It is evident that the TPD flow \correction{is asymptotically stable and the Lyapunov function \eqref{eq: Intro Lyapunov} converges without oscillations}.  

\begin{figure}[htbp]

\begin{subfigure}[t]{0.45\textwidth}
\centering
\includegraphics[width=\textwidth]{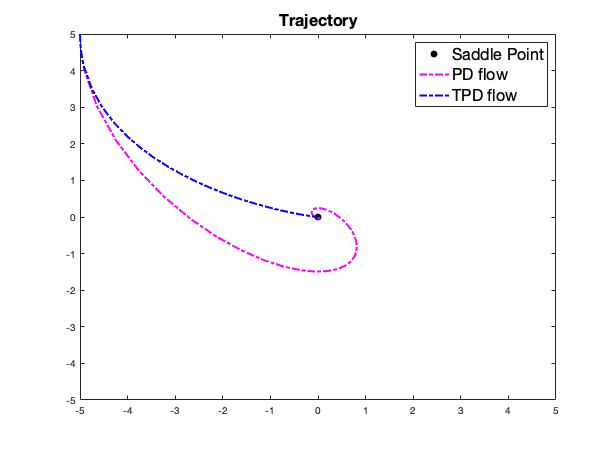}
\caption{Trajectory of PD and TPD flows in the $(u,p)$ coordinate.}
\end{subfigure}
\hfill
\begin{subfigure}[t]{0.45\textwidth}
\centering
\includegraphics[width=\textwidth]{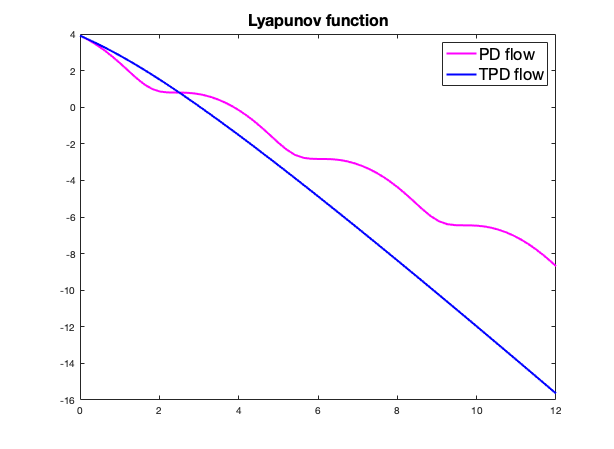}
\caption{Decay of Lyapunov function~\eqref{eq: Intro Lyapunov}.}
\end{subfigure}

\caption{Comparison of PD flow $\begin{pmatrix}
 u'\\
 p'
\end{pmatrix}
= 
\begin{pmatrix}
 - 1 & -1\\
 1 & 0
\end{pmatrix}
\begin{pmatrix}
 u\\
 p
\end{pmatrix}
$ and TPD flow $\begin{pmatrix}
 u'\\
 p'
\end{pmatrix}
= 
\begin{pmatrix}
 - 1 & -1\\
 0 & -1
\end{pmatrix}
\begin{pmatrix}
 u\\
 p
\end{pmatrix}
$ for $\mathcal L(u,p) = \frac{1}{2}u^2 - up$. 
The ODE systems are solved by \mc{ode45} in MATLAB.}
\label{fig:PDvsTPD}
\end{figure}

%
%
On convergence analysis, for linear saddle point systems, it suffices to bound the spectrum of a matrix operator for the error; see~\cite{notay2019convergence, zulehner2002analysis} and reference therein. For nonlinear problems, if the spectrum analysis is applied to the linearization problem, then it is limited to the local convergence, i.e., $(u_k, p_k)$ should be sufficiently close to $(u^*, p^*)$; see, e.g.~\cite{hu2006nonlinear}.

To overcome the limitation of the spectrum analysis, we shall follow the framework in~\cite{chen2021unified} to verify the strong Lyapunov property in Theorem \ref{Continuous strong Lyapunov for transformed gradient flow}
$$
 -\nabla \mathcal{E}(u,p) \cdot \mathcal{G}(u,p) \geq \mu \, \mathcal{E}(u,p),
$$
where $\mathcal{G}(u,p)$ is the vector field defined on the right hand side of~\eqref{eq: TPD}. Then the exponential decay~\eqref{eq: Intro exponential decay} follows. Convergence analysis relies crucially on the assumption that the Lipschitz constant $L_{f,\IV} < 2$ which can be always satisfied by a rescaling. 


One can further ask the question: can we still have the linear convergence rate if not only $\mu_g =0$ but also $\mu_f = 0$? Recall that, the strong convexity of the dual variable is recovered by the transformation on the dual variable flow. We can apply the transformation to the primal variable as well. If $f$ is not strongly convex, but $f_B(u) = f(u) + \frac{1}{2}(B^{\top}\IP^{-1}Bu, u)$ is strongly convex, we show the exponential stability can be obtained by the symmetric transformed primal-dual (STPD) flow:
\begin{equation}\label{eq: Intro STPD}
\left \{ \begin{aligned}
u' &=   -\IV^{-1}(\partial_u \mathcal{L}(u,p)  + B^{\top}\IP^{-1}\partial_p \mathcal{L}(u,p))\\
p' &= \IQ^{-1}\left ( \partial_p \mathcal{L}(u,p)-B\IU^{-1}\partial_u \mathcal{L}(u,p) \right ) 
\end{aligned} \right. .
\end{equation}
Here we further introduce SPD matrices $\IU, \IP$ for the transformation and treat $\IV$ and $\IQ$ as preconditioners. 

With appropriate scaling of $\IU$ and $\IP$, we can assume Lipschitz constants $L_{f, \IU} < 2$ and  $L_{g, \IP} < 2$. Then define the effective convexity constant $\mu = \min \{\mu_{\mV}, \mu_{\mQ}\}$ with
$$
\mu_{\mV} = \min \{1, 2 - L_{f, \IU}\} \mu_{f_B, \IV}, \quad \mu_{\mQ} = \min \{1, 2 - L_{g, \IP}\} \mu_{g_B, \IQ},
$$
in Theorem \ref{thm: strong Lyapunov property for symmetric TPD flow}, we show the exponential decay
$$\mathcal{E}(u(t),p(t)) \leq e^{-\mu t}\mathcal{E}(u(0),p(0)), \quad \forall t > 0,$$
for $(u(t), p(t))$ solves the STPD flow~\eqref{eq: Intro STPD}. 
%

Consider the convex optimization problems with affine equality constraints~\eqref{eq: one-block affine equality constrained optimization system}, the well-known augmented Lagrangian method (ALM)~\cite{hestenes1969multiplier, powell1969method} for solving
\begin{equation}\label{eq: Intro AL}
   \min_{u \in \mathbb{R}^m} \max_{p \in \mathbb{R}^n} \mathcal{L}_\beta (u,p) = f(u) + \frac{\beta}{2} \|Bu-b\|^2 + (p, Bu-b)
\end{equation} 
can be derived from STPD flow~\eqref{eq: Intro STPD} by choosing $\IP^{-1}=\beta I$. From this point of view, the effectivness of ALM can be interpreted by the STPD flows in the continuous level. Notice we can also consider TPD flow for the augmented Lagrangian~\eqref{eq: Intro AL}  which is more or less equivalent to STPD~\eqref{eq: Intro STPD} for the original Lagrangian. We show careful analysis to explain the connection between TPD flows and ALM in Section \ref{sec:ALM}.

To illustrate different flows for constrained optimization problems~\eqref{eq: one-block affine equality constrained optimization system}, we present numerical results in Fig. \ref{fig:PDvsTPD_ALM} for the example 
\begin{equation}\label{eq: example: one-block affine equality constrained optimization system}
    \begin{aligned}
    \min_{(u_1, u_2)\in \mathbb{R}^2} f(u_1, u_2) = \frac{1}{2}u_1^2 - u_2\\
    \text{subject to} \quad u_1 - u_2 = 0.
    \end{aligned}
\end{equation}
with $u = (u_1, u_2)\in \mathbb R^2, p\in \mathbb R$. The convex function $f$ is not strongly convex but restricted to $\ker B = \{ (u_1, u_2)\in \mathbb R^2: u_1= u_2\}$ is or equivalently $f_B(u_1, u_2) = \frac{1}{2} u_1^2 + \frac{1}{2}(u_1-u_2)^2 - u_2$ is strongly convex. 
Compared with applying the PD flow to Lagrangian (PD flow) or augmented Lagrangian (AL-PD flow), \correction{the STPD flow approached the saddle point with no oscillation and dramatic decay of the Lyapunov function \eqref{eq: Intro Lyapunov}}. 

\begin{figure}[htbp]
\begin{subfigure}[t]{0.45\textwidth}
\centering
\includegraphics*[width=\textwidth]{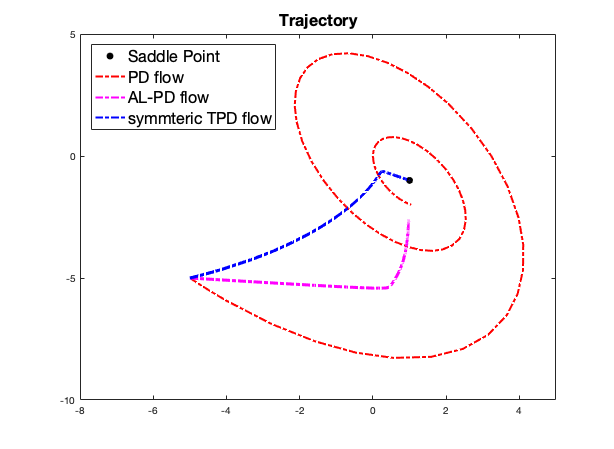}
\caption{Trajectories of PD, AL-PD and STPD flows in $(u_1, p)$ coordinate.}
\end{subfigure}
\hfill
\begin{subfigure}[t]{0.45\textwidth}
\centering
\includegraphics*[width=\textwidth]{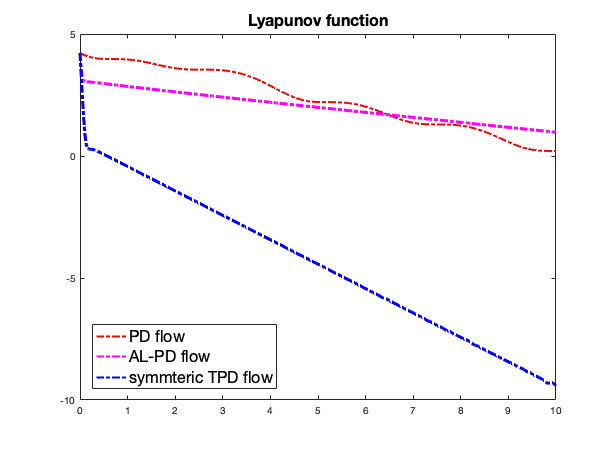}
\caption{Decay of the Lyapunov function~\eqref{eq: Intro Lyapunov}. }
\end{subfigure}

\caption{Comparison of PD flow, AL-PD flow and STPD flow for the example~\eqref{eq: example: one-block affine equality constrained optimization system}. In STPD, $\IU = \IV = I$ and  $\IP^{-1} = \IQ^{-1} = \beta I$ with $\beta = 10$. The ODE systems are solved by \mc{ode45} in MATLAB.}
\label{fig:PDvsTPD_ALM}
\end{figure}


\subsection{Schemes}
In the discrete level, we apply implicit Euler, explicit Euler, implicit-explicit (IMEX) methods, \correction{and a Gauss-Seidel iteration with accelerated overrelaxation (AOR)~\cite{hadjidimos1978accelerated}} to the TPD flow~\eqref{eq: TPD} to obtain several iterative methods. 

Implicit Euler method with growing step size and efficient Newton type inner iteration~\cite{Li;Sun;Toh:2018highly} will yield super-linear convergence rate. 
On the explicit Euler method, an equivalent algorithm  is:  
\begin{equation}
   \begin{aligned}\label{eq: Intro EE TPD}
u_{k+1/2} &= u_k - \IV^{-1}(\nabla f(u_k) + B^\top p_{k}),\\ 
p_{k+1}  &= p_k - \alpha_k \IQ^{-1}\left(\nabla g(p_k) - B u_{k+1/2} \right),\\
u_{k+1} &= (1-\alpha_k)u_k + \alpha_k u_{k+1/2},
\end{aligned} 
\end{equation}
which can be viewed as a relaxation of the inexact Uzawa methods (IUM) and recovers IUM when $\alpha_k=1$. The term $u_{k+1/2}$ is introduced for computing $B\IV^{-1}\partial_u \mathcal{L}(u_k,p_k)$ in~\eqref{eq: TPD}. In other words, TPD flow can be viewed as a continuous version of IUM by dividing $\alpha_k$ and letting $\alpha_k \to 0$ in~\eqref{eq: Intro EE TPD}. 


When the step size $\alpha_k$ is sufficiently small, in Theorem \ref{convergence rate for EE discretization}, we prove that
%
$$\mathcal{E}(u_{k+1}, p_{k+1}) \leq (1-\frac{1}{4\kappa^2})\mathcal{E}(u_k, p_k),$$
with $\kappa \geq \max\{ \kappa_{\mV}, \kappa_{\mQ}\}, \kappa_{\mV} := L_{\mV}/\mu_{\mV}, \kappa_{\mQ} := L_{\mQ}/\mu_{\mQ}$. We refer to Table \ref{table:constants} for the precise definition of these constants and comment on the rate briefly here. 

Roughly speaking, the rate of convergence is determined by $\kappa_{\mV}(f):= L_{f, \IV}/\mu_{f, \IV}$ and  $\kappa_{\mQ}(S) = \kappa(\IQ^{-1}B\IV^{-1}B^\top):= \lambda_{\max}\left (\IQ^{-1}B\IV^{-1}B^\top \right)/\lambda_{\min}\left (\IQ^{-1}B\IV^{-1}B^\top \right)$. Both $\IV$ and $\IQ$ can be scalar, then~\eqref{eq: Intro EE TPD} is an explicit first order method with linear convergence rate. However, in this case, when either $\kappa(f)$ or $\kappa(BB^{\top})$ is large, the convergence will be very slow. When $\IV^{-1} = 1/L_f I$, we can choose $\IQ^{-1} = L_f (BB^{\top})^{-1}$ to improve $\kappa_{\mQ}$ and the rate becomes $1-c/\kappa^2(f)$.

To further accelerate the linear rate $1-c/\kappa^2$, we consider the IMEX scheme for TPD flow~\eqref{eq: TPD}. Equivalently we replace the third step in~\eqref{eq: Intro EE TPD} by
\begin{equation}\label{eq:intro subproblem}
u_{k+1} = \arg \min_{u \in \mathbb{R}^m} f(u) + \frac{1}{2\alpha_k} \|u - u_k + \alpha_k\IV^{-1}B^\top p_{k+1}\|^2_{\IV}. 
\end{equation}
When $\IV =L_f I$,~\eqref{eq:intro subproblem} is one proximal iteration
$$
u_{k+1} = {\rm prox}_{ f, \frac{\alpha_k}{L_f}}  (u_k - \frac{\alpha_k}{L_f} B^\top p_{k+1}),
$$
where recall that ${\rm prox}_{f, \lambda}(w) = \arg \min_{u} f(u) + \frac{1}{2\lambda}\|u-w\|^2$. Namely IMEX for~\eqref{eq: TPD} is equivalent to one inexact Uzawa iteration plus one proximal iteration. 
The linear convergence rate can be improved to (Theorem \ref{convergence rate for IMEX discretization}), 
\begin{equation}\label{intro:IMEX}
\mathcal{E}(u_{k+1}, p_{k+1}) \leq \frac{1}{1 + c/\kappa_{\mV}}\mathcal{E}(u_k, p_k),
\end{equation}
provided we can choose $\IQ$ such that $\kappa_{\mQ}(S) \ll \kappa_{\mV}$. 
We can choose an inner product $\IV$ so that $\kappa_{\mV}(f)$ small. But in the above schemes a prior information on the spectrum of the Schur complement $B\IV^{-1}B^\top$ is required to design $\IQ$ in order to control $\kappa_{\mQ}(S)$. Noted that when $\IV^{-1} = A^{-1}$ is a dense matrix, even the Schur complement $B\IV^{-1}B^\top$ is expensive to compute and store. When the proximal operator of $f$ is available, we recommend $\IV =L_f I$ and $\IQ^{-1} \approx L_f(BB^{\top})^{-1}$ so that~\eqref{intro:IMEX} can be achieved. \correction{In particular, $\IV = rI$ and $\IQ =\frac{1}{r}BB^{\top} +\delta I$ is the scheme discussed in~\cite{he2021balanced} and a sub-linear rate of $1/k$ is given for (non-smooth) constrained problems there.}
 
\correction{When the proximal operator of $f$ is not available, we propose a new Gauss-Seidel iteration with accelerated overrelaxation (GS-AOR) for the TPD flow:
\begin{equation}\label{eq:introG-S TPD}
\begin{aligned}
  \frac{u_{k+1} - u_k}{\alpha}&= -\IV^{-1}(\nabla f(u_{k}) + B^{\intercal}p_{k}) \\
\frac{p_{k+1} - p_k}{\alpha} &=-\IQ^{-1}\left [ \nabla g_B(p_{k})- B(2u_{k+1} - u_k) + B\IV^{-1}\nabla f(u_{k+1}) \right ].     
\end{aligned}
\end{equation}
This is an explicit scheme due to the update of $u_{k+1}$ before the update of $p_{k+1}$. The term $Bu$ in \eqref{eq: TPD} is approximated by $B(2u_{k+1} - u_k)$. With a modified Lyapunov function 
\begin{equation*}
    \mathcal{E}(x_k) = \frac{1}{2}\|x_k-x^{
    \star}\|^2_{\mathcal M_{\mathcal X}- 2\alpha\mathcal B}  -\alpha D_f(u^*, u_k) - \alpha D_{g_B}(p^*, p_{k}),
\end{equation*}
where  $x = (u, p)$, $ \mathcal M_{\mathcal X} = \text{diag}\{\IV, \IQ\}$ and $\mathcal B = \begin{pmatrix}
       0&  B^{\top}\\
       B& 0
\end{pmatrix}$ is a symmetric matrix, and the Bregman divergence of $f$ and $g_B$ are
\begin{equation*}
\begin{aligned}
    D_f(u,v) &= f(u) - f(v) - \langle \nabla f(v), u-v\rangle, \\
    D_{g_B}(p,q) &= g_B(p) - g_B(q) - \langle \nabla g_B(q), p-q\rangle,
\end{aligned}
\end{equation*}
we proved in Theorem \ref{thm: convergence of G-S TPD} that
\begin{equation*}
\begin{aligned}
\mathcal{E}(x_{k+1})\leq&~ \frac{1}{1+\mu \alpha/2}\mathcal{E}(x_k) \leq \frac{1}{1+c\kappa}  \mathcal{E}(x_k),
\end{aligned}
\end{equation*}
where $\mu = \min\left \{\mu_{\mV},\mu_{\mQ} \right \}$ and a fixed step size $\alpha_k = \alpha < 1/\max \{ 4L_S, 2L_{f,\IV}, 2L_{g_B, \IQ} \}$ with the constants defined in Table \ref{table:constants}. In particular, for the constrained optimization problem \eqref{eq: one-block affine equality constrained optimization system}, with a large enough $\IQ$ such that $L_S \leq 1$, constant step size $\alpha = 1/8$ is allowed.
}



\correction{We can combine the transformed primal-dual itertion with the augmented Lagrangian methods. As we mentioned before, $f$ may not be strongly convex but 
\begin{equation*}
    f_\beta(u) = f(u) + \frac{\beta}{2}\|Bu-b\|^2
\end{equation*}
is $\mu_{f_\beta}$-strongly convex. That is, $f$ is strongly convex restricted on $\operatorname{ker} B = \{u\in \mathbb R^m: Bu = 0\}$.} By choosing an appropriate SPD matrix $A$, the condition number of $f$ can be modified to $\kappa_A(f) = L_{f,A}/\mu_{f,A}$. For $\IV = A_\beta = A + \beta BB^\top$, a simple $\IQ^{-1}= \beta I$ is allowed as preconditioning of the Schur complement. \correction{We propose the ALM-GS-AOR scheme
\begin{equation*}
\left \{\begin{aligned}
  \frac{u_{k+1} - u_k}{\alpha}= &-\IV^{-1}(\nabla f(u_{k}) + \beta B^{\top}(Bu_k-b) + B^{\top}p_{k}) \\
\frac{p_{k+1} - p_k}{\alpha} =&-\beta \left [ B\IV^{-1}B^{\top}p_k+b  - B(2u_{k+1} - u_k)\right.\\
&\left. + B\IV^{-1}\left (\nabla f(u_{k+1}) + \beta B^{\top}(Bu_{k+1}-b)\right ) \right ].
\end{aligned} \right.
\end{equation*}}
We show in Proposition \ref{preconditioned Schur complement} that 
$$
\kappa_{\mQ}(S) = \kappa( \IQ^{-1}B\IV^{-1}B^\top) \leq 1 + \frac{1}{\beta \mu_{S_0}},
$$
where $\mu_{S_0} = \lambda_{\min} (BA^{-1}B^{\top})$. So for $\beta$ large enough, e.g., $\beta \geq 1/ \mu_{S_0}$, $\kappa_{\mQ}(S)$ is bounded by $2$. \correction{Then with constant step size $\alpha =1/8$, we get the linear rate \begin{equation*}
\begin{aligned}
\mathcal{E}(x_{k+1})\leq&~ \frac{1}{1+\mu_{f_\beta, A_\beta}/16}\mathcal{E}(x_k) \leq \frac{1}{1+c\, \kappa_{A_\beta}(f_\beta)}  \mathcal{E}(x_k).
\end{aligned}
\end{equation*}
}

The choice $\IQ^{-1} = \beta I_n$ is simple but now $\IV^{-1} \approx (A + \beta BB^\top)^{-1}$ becomes harder to approximate. \correction{General preconditioners $\IV$ and $\IQ$ can be chosen and analyzed under the framework of transformed primal-dual methods, which extends the choice of augmented term parameter is usually a scalar in ALM literatures~\cite{bertsekas2014constrained, powell1978algorithms}.} An optimal choice of parameter $\beta$ and inner product $\IV$ and $\IQ$ will be problem dependent. 
We summarize some typical choices of $\IV$ and $\IQ$ for explicit Euler, IMEX, \correction{and GS-AOR} schemes with or without ALM in Table \ref{table:ALM examples}. 

\subsection{Contribution}
To summarize, our main contribution of this work includes:
\begin{itemize}
    \item We propose a novel transformed primal-dual flow and \correction{prove  the saddle point $(u^*, p^*)$ is exponentially stable by showing} the exponential decay of a strong Lyapunov function. We show the symmetrized version can recover the well-known ALM.  
    
    \item In the discrete level, we develop several transformed primal-dual iterations by applying implicit Euler, explicit Euler, implicit-explicit Euler, and \correction {GS-AOR} methods of the TPD flow. All the schemes achieve the linear convergence rates with mild assumptions, even neither $f$ nor $g$ is strongly convex. \correction{In particular, GS-AOR is an explicit scheme achieving the state-of-the-art linear convergence rate.}
        
    
    \item Instead of solving a subproblem at each iteration accurately, we can relax to general linear inexact solvers $\IV^{-1}$ and $\IQ^{-1}$. We also derive convergence analysis with nonlinear inexact inner solvers for sub-problem~\eqref{eq:intro subproblem}. Compared with existing works, our framework using the strong Lyapunov property provides flexibility and much clear analysis to choose inexact inner solvers. 
\end{itemize}

The rest of paper is organized as follows. In Section 2 we describe problem settings and review Lyapunov analysis used as tools for convergence analysis. Our motivation to use change of variable to recover strong convexity in dual variable is also highlighted in this section. In Section 3, the transformed primal-dual flow on the continuous level is developed and convergence analysis is given. Variants of discrete schemes as transformed primal-dual iterations are discussed in Section 4 and we further generalize our framework to inexact solvers. A \correction{symmetric} transformed primal-dual flow for non-strongly convex $f$ and $g$ is proposed and analyzed in Section 5. In Section 6, we showed our algorithms can be adapted to augmented Lagrangian to solve constrained optimization problems.

\section{Preliminaries}
In this section, we provide background on convex functions and Lyapunov analysis. We also show the loss of exponential stability for the primal-dual flow and recover it by a change of variable.
 
\subsection{Convex Functions}
Let $\mV$ be a finite-dimensional Hilbert space with inner product $(\cdot, \cdot)$ and norm $\|\cdot \|$. $\mV^{\prime}$ is the linear space of all linear and continuous mappings $T: \mV \rightarrow \mathbb{R}$, which is called the dual space of $\mV$, and $\langle \cdot, \cdot \rangle$ denotes the duality pair between $\mV$ and $\mV^{\prime}$. For any proper closed convex function $f: \mV \rightarrow \mathbb{R}$ , we say $f \in \correction{\mathcal{S}_{\mu}}$ with $\mu \geqslant 0$ if $f$ is differentiable and  
$$
f\left(v\right)-f\left(u\right)-\left\langle \nabla f(u), v-u\right\rangle \geqslant \frac{\mu}{2}\left\|u-v\right\|^{2}, \quad \forall u, v \in \mV.
$$
 In addition, denote $f \in \correction{\mathcal{S}_{\mu, L}}$ if $f \in \mathcal{S}_{\mu}$ and there exists $L>0$ such that
$$
f\left(v\right)-f\left(u\right)-\left\langle\nabla f\left(u\right), v-u\right\rangle \leqslant \frac{L}{2}\left\|u-v\right\|^{2}, \quad \forall u, v \in \mV.
$$
The Bregman divergence of $f$ is defined as
$$D_{f}(v, u) : = f(v)-f(u)-\langle \nabla f(u), v-u\rangle.$$
For fixed $u \in \mV, D_{f}(\cdot, u)$ is convex as $f$ is convex. If 
$f \in \mathcal{S}_{\mu, L} $, we have
$$\frac{\mu}{2}\|u-v\|^2 \leq D_{f}(v, u) \leq \frac{L}{2}\|u-v\|^2.$$
Especially for $f(u)=\frac{1}{2}\|u\|^{2}$, Bregman divergence reduces to the half of the squared distance $D_{f}(v, u)=D_{f}(u, v)=$ $\frac{1}{2}\|u-v\|^{2}$. In general $D_f(v, u)$ is non-symmetric in terms of $u$ and $v$.
A symmetrized Bregman \correction{divergence} is defined as
\begin{equation*}\label{eq: symmetrized Bregman divergenve}
    \langle\nabla f(u)-\nabla f(v), u-v\rangle=D_{f}(v, u)+D_{f}(u, v).
\end{equation*}
By direct calculation, we have the following three-terms identity.
\begin{lemma}[Bregman divergence identity~\cite{chen1993convergence}]
If \correction{$f: \mV \rightarrow \mathbb{R}$ is differentiable, then for any $u, v, w \in \mV$}, it holds that
\begin{equation}\label{eq: Bregman divergence identity}
   \langle\nabla f(u)-\nabla f(v), v-w\rangle=D_{f}(w, u)-D_{f}(w, v)-D_{f}(v, u). 
\end{equation}
\end{lemma}

When $f(u) = \frac{1}{2}\|u\|^2$, identity~\eqref{eq: Bregman divergence identity} becomes 
\begin{equation*}
 (u-v, v-w) =  \frac{1}{2}\|w-u\|^2 - \frac{1}{2}\|w-v\|^2 - \frac{1}{2}\|v-u\|^2.
\end{equation*}

\subsection{Lyapunov analysis}
In order to study the stability of an equilibrium $x^*$ of a dynamical system defined by an autonomous system
\begin{equation}\label{autonomous system}
    x' = \mathcal{G}(x(t)),
\end{equation}
Lyapunov introduced the so-called Lyapunov function $\mathcal{E}(x)$~\cite{Khalil:1173048, Haddad2008}, which is nonnegative and the equilibrium point $x^*$ satisfies $\mathcal{E}\left(x^{*}\right)=0$ and the Lyapunov condition:
$-\nabla \mathcal{E}(x) \cdot \mathcal{G}(x)$ is locally positive near the equilibrium point $x^{*}$. 
That is the flow $\mathcal G(x)$ may not be in the perfect $- \nabla \mathcal{E}(x)$ direction but contains positive component in that direction.
Then the (local) decay property of $\mathcal{E}(x)$ along the trajectory $x(t)$ of the autonomous system~\eqref{autonomous system} can be derived immediately
$$
\frac{\mathrm{d}}{\mathrm{d} t} \mathcal{E}(x(t))=\nabla \mathcal{E}(x) \cdot x^{\prime}(t)=\nabla \mathcal{E}(x) \cdot \mathcal{G}(x)<0.
$$
To further establish the convergence rate of $\mathcal{E}(x(t))$, Chen and Luo~\cite{chen2021unified} introduced the strong Lyapunov condition: $\mathcal{E}(x)$ is a Lyapunov function and there exist constant $q \geqslant 1$, strictly positive function $c(x)$ and function $p(x)$ such that
\begin{equation}\label{strong Lyapunov property}
- \nabla \mathcal{E}(x) \cdot \mathcal{G}(x) \geq c(x) \mathcal{E}^{q}(x)+ p^{2}(x)  
\end{equation}
holds true near $x^{*}$. From this, one can derive the exponential decay $\mathcal{E}(x(t))=O\left(e^{-c t}\right)$ for $q=1$ and the algebraic decay $\mathcal{E}(x(t))=O\left(t^{- 1 /(q-1)}\right)$ for $q>1$. Furthermore if $\|x - x^*\|^2 \leq C \mathcal E(x)$, then we can derive the exponential stability of $x^*$ from the exponential decay of Lyapunov function $\mathcal E(x)$. 

Note that for an optimization problem, we have freedom to design the vector field $\mathcal G(x)$ and choose Lyapunov function $\mathcal E(x)$. \correction{Throughout this paper, zeros denote zero numbers or zero vectors that is clear from the context. For example, $\mathcal {G} (x^*) =0$ means a vector zero and $\mathcal{E}(x^*) =0$ means a scalar zero for an equilibrium point $x^{\star}$.}

\subsection{Primal-dual flow}
One of the simplest Lyapunov function for the saddle point system~\eqref{eq: min-max problem} is:
\begin{equation}\label{eq:Eup}
    \begin{aligned}
    \mathcal{E}(u, p) = & \frac{1}{2}\|u-u^*\|^2 +   \frac{1}{2}\|p-p^*\|^2.
    \end{aligned}
\end{equation}
\correction{The asymptotic convergence
properties of the PD flow is discussed in~\cite{cherukuri2017saddle}. We state in the following Lemma  that $\mathcal{E}$ is a Lyapunov function but may not satisfy the  strong Lyapunov property when $g$ is not strongly convex.} 

\begin{lemma}\label{lem: Lyapunov property for primal-dual flow}
Assume $f(u) \in \mathcal{S}_{\mu_f, L_f} $ and  $g(p) \in \mathcal{S}_{\mu_g, L_g}$ with $\mu_f > 0$, $\mu_g \geq 0$. Then it holds that
$$
-\nabla \mathcal{E}(u,p)\cdot 
\begin{pmatrix}
 - \partial_u\mathcal L(u,p)\\
 \partial_p \mathcal L(u,p)
\end{pmatrix}
\geq \mu_f\|u-u^*\|^2 + \mu_g \|p-p^*\|^2 \correction{\geq 0},$$
for $\mathcal E(u,p)$ defined in~\eqref{eq:Eup}.
\end{lemma}
\begin{proof}
As $\nabla \mathcal{L}(u^*,p^*) = 0$, we can insert $\nabla \mathcal{L}(u^*,p^*)$ and obtain
\begin{equation*}
 \begin{aligned}
- \nabla \mathcal{E}(u,p)\cdot 
\begin{pmatrix}
 - \partial_u\mathcal L(u,p)\\
 \partial_p \mathcal L(u,p)
\end{pmatrix} ={}& \langle \partial_u \mathcal{E}(u, p), \partial_u \mathcal{L}(u,p) - \partial_u \mathcal{L}(u^*,p^*) \rangle \\
   &+ \langle \partial_p \mathcal{E}(u, p), -\partial_p \mathcal{L}(u,p) + \partial_p \mathcal{L}(u^*,p^*) \rangle \\
    = {}&\langle u-u^*, \nabla f(u) - \nabla f(u^*) \rangle + \langle p-p^*, \nabla g(p) - \nabla g(p^*) \rangle \\ 
    \geq {}& \mu_f\|u-u^*\|^2 + \mu_g \|p-p^*\|^2.
    \end{aligned}
\end{equation*}
\end{proof}

By sign change of $ \partial_u \mathcal{L}(u,p)$ and $ \partial_p \mathcal{L}(u,p)$, the cross terms $\left \langle u-u^*, B^\top(p-p^*) \right \rangle$ and $\left \langle p-p^*, -B(u-u^*) \right \rangle$ are canceled. The symmetrized Bregman divergence of $f$ can be bounded below by $\|u-u^*\|^2 $ by the strong convexity of $f(u)$. However, that of $g$ cannot be controlled by $\|p-p^*\|^2$ if $\mu_g = 0$, which is the loss of the strong convexity on the dual variable. One cannot achieve the exponential decay for Lyapunov function~\eqref{eq:Eup} by using the primal-dual flow, and this is the essential reason for the sub-linear convergence rate for many numerical schemes; see the literature review in the introduction. 

In the continuous level, a compensation is to introduce a rescaled primal-dual flow \correction{and design a tailored Lyapunov function such that the exponential decay can be verified under certain metric~\cite{chen2021unified, qu2018exponential}}.
In the discrete level, however, corresponding explicit schemes can only converge sub-linearly~\cite{luo2022primal}. The linear rate can be retained if the scheme is implicit in $p$~\cite{luo2021accelerated, luo2022primal} for which a linear saddle point system should be solved in each step. 
Recovery the strong Lyapunov property through the time rescaling in the dual variable is thus expensive.

\subsection{Recovery \correction{of} strong convexity through transformation}

In view of~\eqref{eq: dual problem}, when $f^*$ is known, the flow for the dual variable can be the gradient flow of the strong convex function of the dual variable~\cite{huang2013accelerated, yin2010analysis}. In general, we consider a change of variable
\begin{equation}\label{new variable}
    v = u + \IV^{-1} B^\top p.
\end{equation}
After transformation, the optimization problem can be formulated in terms of $(v, p)$, i.e., $ 
\mathcal{L}(v,p):= \mathcal{L}(u(v,p),p).$ Such idea has been successfully applied to the linear saddle point systems in~\cite{benzi2006augmented, chen2017convergence}.  
The primal-dual flow for $(v,p)$ is
 \begin{equation}\label{primal-dual flow in (v,p)}
 \left\{\begin{aligned}
 v' &=  -\partial_v \mathcal{L}(v,p) = -\partial_u \mathcal{L}(u,p),\\
 p'  &= \partial_p \mathcal{L}(v,p) = \partial_p \mathcal{L}(u,p)-B\mathcal{I}_{\mV}^{-1}\partial_u \mathcal{L}(u,p),
 \end{aligned}\right.
 \end{equation}
which can be rewritten as the iteration of $(u,p,v)$ variable
\begin{equation*}
\left\{\begin{aligned}
v ' &=  -v + e(u),\\
p'  &=  - \nabla g_B(p) + B e(u),
\end{aligned}\right.
\end{equation*}
where $e(u) = u - \IV^{-1}\nabla f(u)$ and $g_B(p) = g(p)  + \frac{1}{2}(B\IV^{-1}B^\top p, p).$ 
If $f(u) = \frac{1}{2}\|u\|_A^2$ is quadratic and $\IV = A$, the term $e(u)$ vanishes, then $v'=-v$ and $p'= - \nabla g_B(p)$ is decoupled for which the exponential decay can be easily obtained. 

In general, we can show if $e(u)$ is a contraction, the strong Lyapunov property can be established for the primal-dual flow~\eqref{primal-dual flow in (v,p)} for variable $(v,p)$. In Section \ref{sec: transformed primal-dual flow}, we shall present a simplified flow for the original variable $(u,p)$. 

\subsection{Inner products} 
When $\mV = \mathbb R^m, \mQ= \mathbb R^n$, the standard $l^{2}$ dot product of Euclidean space is usually chosen as the inner product and the norm induced is the Euclidean norm. We now introduce inner product $(\cdot, \cdot)_{\IV}$ induced by a given SPD operator $\IV: \mV\to \mV$ defined as follows
$$
(u, v)_{\IV}:=(\IV u, v) = (u, \IV v), \quad \forall u, v \in \mV
$$ and associated norm $\|\cdot\|_{\IV}$, given by
$$
\|u\|_{\IV}= (u, u)^{1/2}_{\IV}.
$$
The dual norm w.r.t the $\IV$-norm is defined as: for $\ell \in \mV^{\prime}$
$$
\|\ell\|_{\mV^{\prime}}=\sup _{0 \neq u \in \mV} \frac{\langle\ell, u\rangle}{\|u\|_{\IV}}.
$$
It is straightforward to verify that
$$
\|\ell\|_{\mV^{\prime}}=\|\ell\|_{\IV^{-1}}:=\left( \ell, \ell\right )_{ \IV^{-1}}^{1/ 2} := \left( \IV^{-1} \ell, \ell\right )^{1/ 2}.
$$

We shall generalize the convexity and Lipschitz continuity with respect to $\IV$-norm: we say $f \in \mathcal{S}_{\mu_{f, \IV}}$ with $\mu_{f, \IV} \geqslant 0$ if $f$ is differentiable and
$$
f\left( v \right)-f\left(u\right)-\left\langle \nabla f(u), v-u\right\rangle \geqslant \frac{\mu_{f, \IV}}{2}\left\|u-v\right\|^{2}_{\IV}, \quad \forall u, v  \in \mV.
$$
In addition, denote $f \in \mathcal{S}_{\mu_{f, \IV}, L_{f, \IV}}$ if $f \in \mathcal{S}_{\mu_{f, \IV}}$ and there exists $L_{f, \IV}>0$ such that
$$
f\left( v \right)-f\left(u\right)-\left\langle \nabla f(u), v-u\right\rangle \leq \frac{L_{f, \IV}}{2}\left\|u-v\right\|^{2}_{\IV}, \quad \forall u, v  \in \mV.
$$
Under this definition, the default norm is a special case with $\IV = I$ for which the subscript will be skipped, i.e., $\mu_f, L_f$ for $\|\cdot \|$. 
Similarly we introduce inner product $(\cdot, \cdot)_{\IQ}$ induced by a given self-adjoint and positive definite operator $\IQ$ and the notation follows on $\mQ$. The convexity and Lipschitz constant of $g$ w.r.t to $\| \cdot \|_{\IQ}$ will be denoted by $\mu_{g, \IQ}$ and $L_{g, \IQ}$. 

\subsection{Gradient descent step for the primary variable}\label{sec: eu}
For $f \in \mathcal{S}_{\mu_{f,\IV}L_{f,\IV}}
$, function
\begin{equation}\label{eq:eu}
    \begin{aligned}
    e(u) &= u - \IV^{-1}\nabla f(u)
    \end{aligned}
\end{equation}
can be thought of as one gradient descent step at $u$ in the metric $\IV$. 
By the triangle inequality, $e(u)$ is always Lipschitz continuous with respect to $\IV$-norm. 
%
Denote by $L_{e, \IV}$ the Lipschitz constant of $e(u)$, i.e.,  $L_{e, \IV} > 0$ such that
\begin{equation*}
    \begin{aligned}
    \|e(u_1) - e(u_2)\|_{\IV} \leq L_{e, \IV} \|u_1-u_2\|_{\IV},\quad \forall u_1,u_2 \in \mV.  
    \end{aligned}
\end{equation*}
When $L_{e, \IV} < 1$, $e(u)$ is a contractive map. We derive a sufficient and necessary condition for $e(u)$ being contractive in the following lemma.

\begin{lemma}\label{lem: Equivalence between Le and Lf} Suppose $f \in \mathcal{S}_{\mu_{f,\IV}L_{f,\IV}}$. Then $L_{e, \IV} < 1$ if and only if $0<L_{f, \IV} < 2$. 
\end{lemma}
\begin{proof}
Consider $u_1, u_2 \in \mV$,
\begin{equation}\label{e(u1) - e(u2) square}
    \begin{aligned}
    \|e(u_1) - e(u_2)\|_{\IV }^2 ={}& \|u_1 - u_2 - \IV^{-1}(\nabla f(u_1) - \nabla f(u_2))\|_{\IV}^2 \\
     ={}& \|u_1-u_2\|_{\IV}^2+ \|\nabla f(u_1) - \nabla f(u_2)\|_{\IV^{-1}}^2 \\
    &- 2\langle u_1-u_2, \nabla f(u_1) - \nabla f(u_2) \rangle .
    \end{aligned}
\end{equation}
If $L_{e, \IV} < 1$, we have $\|e(u_1) - e(u_2)\|_{\IV }^2  < \|u_1-u_2\|_{\IV}^2, $
and by~\eqref{e(u1) - e(u2) square}
\begin{equation*}
    \begin{aligned}
    \|\nabla f(u_1) - \nabla f(u_2)\|_{\IV^{-1}}^2 &< 2\langle u_1-u_2, \nabla f(u_1) - \nabla f(u_2) \rangle\\
    & \leq 2 \|\nabla f(u_1) - \nabla f(u_2)\|_{\IV^{-1}} \|u_1-u_2\| _{\IV},
    \end{aligned}
\end{equation*}
which implies $L_{f, \IV} < 2$. If $L_{f,\IV} =0$, then $\|e(u_1) - e(u_2)\|_{\IV }^2  =\|u_1-u_2\|_{\IV}^2$ contradicts with $L_{e, \IV} < 1$. 

We now show sufficiency. If $0< L_{f, \IV} < 2$, then for $u_1, u_2 \in \mV$, we have the inequality~\cite[Chapter 2]{nesterov2003introductory}
$$\|\nabla f(u_1) - \nabla f(u_2)\|_{\IV^{-1}}^2 < 2\langle u_1-u_2, \nabla f(u_1) - \nabla f(u_2) \rangle,$$
and, by~\eqref{e(u1) - e(u2) square},
$$\|e(u_1) - e(u_2)\|_{\IV }^2 < \|u_1-u_2\|^2,$$
which implies $L_{e, \IV} < 1.$

%
%

\end{proof}

The condition $L_{f,\IV} > 0$ is to eliminate the degenerate case $f(u)$ is affine.  The condition $L_{f,\IV}<2$ can be achieved by either a rescaling of $f$ or the inner product $\mathcal I_{\mV}$. For example, for $f \in \mathcal{S}_{\mu_f, L_f}$, we can choose  $\IV^{-1} = \frac{1}{L_f}  I_m < \frac{2}{L_f}I_m$ , then
$$
    \|\nabla f(u_1) - \nabla f(u_2)\|^2_{\IV^{-1}} = \frac{1}{L_f} \|\nabla f(u_1) - \nabla f(u_2)\|^2 \leq  L_f \|u_1-u_2\|^2 = \|u_1-u_2\|^2_{\IV},
$$
for all $u_1, u_2 \in \mV$ which implies $L_{f,\IV}\leq 1$. For this example, the function $e(u)$ is simply a gradient descent step at $u$ for function $f$ with step size $1/L_f$.


\begin{table}[htp]
	\centering
	\caption{Derived convexity constants and Lipschitz constants for $f \in \mathcal{S}_{\mu_{f, \IV}, L_{f, \IV}}$, $g_B \in \mathcal{S}_{\mu_{g_B, \IQ}, L_{g_B, \IQ}}$, with $g_B(p) = g(p)  + \frac{1}{2}(B\IV^{-1}B^\top p, p)$, and $e(u)=u - \IV^{-1}\nabla f(u)$ is Lipschitz continuous with constant $L_{e, \IV}<1$.}
	\renewcommand{\arraystretch}{1.5}
	\begin{tabular}{@{} l l @{}}
	\toprule
	\hskip 1.8cm $\mu$ &  \hskip 1.8cm $L$   \\
 \hline
$\mu_S = \lambda_{\min} \left (\IQ^{-1}B\IV^{-1}B^\top \right)$ & $L_S^2= \lambda_{\max} \left (\IQ^{-1}B\IV^{-1}B^\top \right)$ \\
$\mu_{\mV} = \mu_{f,\IV}$ & $L_{\mV}^2 = 2\left (L_{e, \IV}^2 (1 + L_S^2) \right)$  \\
$\mu_{\mQ} = \left (2- L_{f,\IV}\right) \mu_{g_B, \IQ}$ & $L_{\mQ}^2 = 2 L_{g_B, \IQ}^2$ 
\medskip \\

\bottomrule
	\end{tabular}
	\label{table:constants}
\end{table}
%

\section{Transformed Primal-Dual Flow}\label{sec: transformed primal-dual flow}

In this section, we propose a transformed primal-dual flow and verify the strong Lyapunov property for a quadratic and convex Lyapunov function. Furthermore, we show the Lipschitz continuity of the flow. We assume $f$ is strongly convex but $g$ may not. In view of the dual problem~\eqref{eq: dual problem}, the saddle point $(u^*,p^*)$ exists and is unique. 

\subsection{Transformed primal-dual flow}
Given an SPD matrix $\IV$ for $\mV$ and $\IQ$ for $\mQ$, we consider a  transformed primal-dual flow:
\begin{equation}\label{eq: transformed primal-dual flow}
\left \{ \begin{aligned}
u' &=  \mathcal{G}^u(u,p)\\
p' &= \mathcal{G}^p(u,p)
\end{aligned} \right.
\end{equation}
with
\begin{align}
\label{eq:Gu}  
    \mathcal{G}^u(u,p) &= -\IV^{-1}\partial_u \mathcal{L}(u,p) = -\IV^{-1}(\nabla f(u)+ B^{\top}p) \correction{= e(u) - v}, \\
 \label{eq:Gp}    \mathcal{G}^p(u,p) &= \IQ^{-1}\left ( \partial_p \mathcal{L}(u,p)-B\IV^{-1}\partial_u \mathcal{L}(u,p) \right ) = - \IQ^{-1}\left(\nabla g_B(p) - B e(u)  \right),
\end{align}
where recall that $e(u) = u - \IV^{-1}\nabla f(u), v = u + \IV^{-1} B^\top p$, and $g_B(p) = g(p)  + \frac{1}{2}(B\IV^{-1}B^\top p, p).$ Namely for the primary variable $u$, we use a preconditioned gradient flow and for the dual variable $p$, we use a preconditioned gradient flow associated to $g_B$ but perturbed by $Be(u)$. 
Since $B$ is surjective, $B\IV^{-1}B^\top$ is always SPD. The non-strongly convex function $g(p)$ is enhanced to a strongly convex function $g_B(p)\in \mathcal{S}_{\mu_{g_B,\IQ}, L_{g_B,\IQ}}$. 

We denote $\mathcal{G}(u,p) = (\mathcal{G}^u(u,p), \mathcal{G}^p(u,p))^\top$. The equilibrium point $(u^*, p^*)$ of the flow gives $\mathcal{G}(u^*,p^*) = 0$, which satisfies \correction{the first order condition $\nabla \mathcal{L}(u^*,p^*)=0$}. 


\subsection{Strong Lyapunov property}\label{sec: strong Lyapunov property}
Define Lyapunov function
\begin{equation}\label{eq: Lyapunov}
    \mathcal{E}(u,p) = \frac{1}{2} \|u-u^*\|^2_{\IV} + \frac{1}{2} \|p-p^*\|^2_{\IQ}.
\end{equation}
The transformed primal-dual flow~\eqref{eq: transformed primal-dual flow} satisfies the error equation
\begin{equation*}
\begin{pmatrix}
u-u^* \\
p-p^* 
\end{pmatrix}' =
\begin{pmatrix}
\mathcal{G}^u(u,p) - \mathcal{G}^u(u^*,p^*) \\
\mathcal{G}^p(u,p) - \mathcal{G}^p(u^*,p^*)
\end{pmatrix}.
\end{equation*}
We aim to verify the strong Lyapunov property to obtain the exponential decay. The key is the following lower bound of the cross term. 

\begin{lemma}\label{lem:key}
Suppose $f \in \mathcal{S}_{\mu_{f,\IV}, L_{f,\IV}}$. 
For any $u_1,u_2\in \mV$ and $p_1, p_2\in \mQ$, we have
\begin{align*}
  & \langle \nabla f(u_1)-\nabla f(u_2), \IV^{-1}B^\top(p_1 - p_2) \rangle \\
    \geq{} &\frac{\mu_{f,\IV}}{2} \|v_1-v_2\|^2_{\IV} - \frac{L_{f,\IV}}{2}\| B^{\top}(p_1 - p_2)\|^2_{\IV^{-1}} - \frac{1}{2}\langle \nabla f(u_1)-\nabla f(u_2), u_1 - u_2 \rangle,
\end{align*}
where recall that $v = u + \IV^{-1}B^{\top} p$ is the transformed variable.
\end{lemma}
\begin{proof}
 To use the strong convexity of $f$, we switch between variables using relation $v = u + \IV^{-1} B^\top p$. Write $$\IV^{-1}B^\top(p_1 - p_2) = v_1-v_2-(u_1-u_2) = u_2 - (u_1-v_1+v_2).$$
Using the Bregman divergence identity~\eqref{eq: Bregman divergence identity} and bounds on the Bregman divergence
\begin{equation}\label{triple Bregman divergenve 1}
    \begin{aligned}
    & \langle \nabla f(u_1)-\nabla f(u_2), u_2 - (u_1-v_1+v_2) \rangle\\
     = {}& D_f(u_1-v_1+v_2, u_1) - D_f(u_1-v_1+v_2, u_2) - D_f(u_2, u_1)\\
     \geq {}& \frac{\mu_{f,\IV}}{2}\|v_1-v_2 \|_{\IV}^2 - \frac{L_{f,\IV}}{2}\|u_1 - u_2 - (v_1 - v_2)\|^2_{\IV} - D_f(u_2, u_1)\\
     = {}& \frac{\mu_{f,\IV}}{2}\|v_1-v_2 \|_{\IV}^2 - \frac{L_{f,\IV}}{2}\| B^{\top}(p_1- p_2)\|^2_{\IV^{-1}} - D_f(u_2, u_1).
    \end{aligned}
\end{equation}
Similarly, we exchange $u_1$ and $u_2$ to obtain
\begin{equation}\label{triple Bregman divergenve 2}
    \begin{aligned}
    & \langle \nabla f(u_2)-\nabla f(u_1), u_1 - (u_2+v_1-v_2) \rangle\\
 = {} &D_f(u_2+v_1-v_2, u_2)  - D_f(u_2+v_1-v_2, u_1) - D_f(u_1, u_2)\\
\geq {}& \frac{\mu_{f,\IV}}{2}\|v_1-v_2 \|_{\IV}^2 - \frac{L_{f,\IV}}{2}\| B^{\top}(p_1- p_2)\|^2_{\IV^{-1}} - D_f(u_1, u_2).
    \end{aligned}
\end{equation}
Summing~\eqref{triple Bregman divergenve 1} and~\eqref{triple Bregman divergenve 2}, we obtain the desired bound.

\end{proof}

We next verify the strong Lyapunov property. 
\begin{theorem}\label{Continuous strong Lyapunov for transformed gradient flow}
Assume $f(u) \in \mathcal{S}_{\mu_{f,\IV}, L_{f,\IV} } $ with $0< \mu_{f,\IV} \leq L_{f, \IV}< 2$. 
Then for the Lyapunov function~\eqref{eq: Lyapunov} and the TPD field $\mathcal G$~\eqref{eq:Gu}-\eqref{eq:Gp}, the following strong Lyapunov property holds
\begin{equation}\label{eq: strong lyapunov in continuous level}
    -\nabla \mathcal{E}(u,p) \cdot \mathcal{G}(u,p) \geq \mu \, \mathcal{E}(u,p) + \frac{\mu_{f,\IV}}{2}\| v - v^* \|_{\IV}^2,
\end{equation}
where $0 < \mu = \min \left \{\mu_{\mV}, \mu_{\mQ}\right\}$ with $\mu_{\mV}, \mu_{\mQ}$ defined in Table \ref{table:constants}. Consequently if $(u(t), p(t))$ solves the TPD flow~\eqref{eq: transformed primal-dual flow}, we have the exponential decay
$$\mathcal{E}(u(t),p(t)) \leq e^{-\mu t}\mathcal{E}(u(0),p(0)), \quad t > 0.$$
\end{theorem}

\begin{proof}
To verify the strong Lyapunov property for $\mathcal E(u,p)$, we split it as
\begin{equation*}
    \begin{aligned}
    -\nabla \mathcal{E}(u,p) \cdot \mathcal{G}(u,p) =& -\nabla \mathcal{E}(u,p) \cdot (\mathcal{G}(u,p)-\mathcal{G}(u^*,p^*)) \\
    ={} &\langle u-u^*,  \partial_u \mathcal{L}(u,p)  - \partial_u \mathcal{L}(u^*,p^*) \rangle \\
    &+\langle p-p^*, B\IV^{-1}(\partial_u \mathcal{L}(u,p)  - \partial_u \mathcal{L}(u^*,p^*))\rangle \\
    &- \langle p-p^*, \partial_p \mathcal{L}(u,p)  - \partial_p \mathcal{L}(u^*,p^*)\rangle \\
    :={} & {\rm I}_1+ {\rm I}_2 - {\rm I}_3. 
    \end{aligned}
\end{equation*}
By Lemma \ref{lem: Lyapunov property for primal-dual flow} for the primal-dual flow
\begin{equation*}
    \begin{aligned}
    {\rm I}_1 -{\rm I}_3 
     &=  \langle  \nabla f(u) - \nabla f(u^*), u-u^*\rangle + \langle \nabla g(p)- \nabla g(p^*), p-p^*\rangle,
    \end{aligned}
\end{equation*}
which are non-negative terms.

As $\IV$ and $B$ are linear operators,
\begin{equation*}
    \begin{aligned}
    {\rm I}_2  &= \langle \IV^{-1}B^\top(p-p^*),  \partial_u \mathcal{L}(u,p) - \partial_u \mathcal{L}(u^*,p^*) \rangle \\
    &= \langle \nabla f(u)-\nabla f(u^*), \IV^{-1}B^\top(p-p^*) \rangle + \| B^\top(p-p^*)\|_{\IV^{-1}}^2.
    \end{aligned}
\end{equation*}
We apply Lemma \ref{lem:key} to the cross term $\langle \nabla f(u)-\nabla f(u^*), \IV^{-1}B^\top(p-p^*) \rangle$ to get 
\begin{equation*}
    \begin{aligned}
    &-\nabla \mathcal{E}(u,p) \cdot \mathcal{G}(u,p) -  \frac{\mu_{f,\IV}}{2}\| v - v^* \|_{\IV}^2 \\
    \geq{} & \frac{1}{2}\langle \nabla f(u)-\nabla f(u^*), u-u^* \rangle +\langle \nabla g(p)- \nabla g(p^*), p-p^*\rangle \\
    &+ \left (1-\frac{L_{f,\IV}}{2} \right) \| B^\top(p-p^*) \|_{\IV^{-1}}^2 \\
    \geq{} & \frac{\mu_{\mathcal V}}{2} \|u-u^*\|_{\IV}^2 +  \frac{\mu_{\mathcal Q}}{2}\|p-p^*\|_{\IQ}^2.
    \end{aligned}
\end{equation*}
We then complete the proof by rearranging the terms.
\end{proof}

\begin{remark}\rm 
For the linear saddle point system, $A \in \mathbb{R}^{m\times m}$ is SPD, $C \in \mathbb{R}^{n\times n}$ is positive semidefinite, $f(u) = \frac{1}{2}(Au, u) + (a,u)$ and $g(p) = \frac{1}{2}(Cp, p)  + (c,p)$. An ideal choice is  $\IV^{-1} = A^{-1}$ and $\IQ^{-1} = S^{-1} = (BA^{-1}B^\top+C)^{-1}$. Then we have $L_{e,\IV}=0$,  $\mu_{f,\IV} = L_{f,\IV} = \mu_{g_B, \IQ} = L_{g_B, \IQ} = 1$ and thus
$$-\nabla \mathcal{E}(u,p) \cdot \mathcal{G}(u,p) \geq \mathcal{E}(u,p),$$
which yields the exponential decay
$$\mathcal{E}(u(t),p(t)) \leq e^{- t}\mathcal{E}(u(0),p(0)).$$
However, $A^{-1}$ and $S^{-1}$ are not computable in general. The inner product $\IV^{-1}$ and $\IQ^{-1}$ can be thought of as inexact solvers approximating $A^{-1}$ and $S^{-1}$, respectively. $\square$ 
\end{remark}

To guarantee the exponential decay, we require $0<L_{f,\IV}<2$ which is equivalent to $e(u)$ is a contraction \correction{by Lemma \ref{lem: Equivalence between Le and Lf}}. The requirement can be always satisfied by a rescaling. Indeed in later analysis, we will choose $\IV$ so that $L_{f,\IV}\leq 1$. Then $\mu = \min \{\mu_{f,\IV}, \mu_{g_B,\IQ} \}$. When $\min \{\mu_{f,\IV}, \mu_{g_B,\IQ} \} \ll \max \{\mu_{f,\IV}, \mu_{g_B,\IQ} \}$, further scaling in $\IV$ or $\IQ$ can be introduced to balance the decay rate for the primal and dual variables. For discrete schemes, the rate will be determined by the condition number which is the ratio of Lipschitz constants and the convexity constants.

So next we show that the vector field $\mathcal G(u,p)$ is Lipschitz continuous and give bounds on Lipschitz constants. 
\begin{lemma}\label{Lipschitz continuity of flow}
Assume $\nabla f$ and $\nabla g_B$ are Lipschitz continuous with Lipschitz constant $L_{f,\IV}$ and $L_{g_B, \IQ}$, respectively. Let $L_{e,\IV}$ be the Lipschitz constant of $e(u)$, then we have
\begin{align}
\label{eq:LipGu} \|\mathcal{G}^u(u_1,p_1) - \mathcal{G}^u(u_2,p_2)\|_{\IV} &\leq L_{e, \IV}\|u_1-u_2\|_{\IV} + \|v_1-v _2\|_{\IV},\\
\label{eq:LipGp} \|\mathcal{G}^p(u_1,p_1) - \mathcal{G}^p(u_2,p_2)\|_{\IQ} &\leq  L_{e, \IV}L_S\|u_1-u_2\|_{\IV} + L_{g_B,\IQ}\|p_1-p_2\|_{\IQ},
\end{align}
for all $u_1, u_2 \in \mV$ and $p_1, p_2 \in \mQ$.
\end{lemma}
\begin{proof}
By the formulation~\eqref{eq:Gu} we have
$$
\mathcal{G}^u(u,p) = e(u) - v.
$$
Consequently 
\begin{equation*}
    \begin{aligned}
    \|\mathcal{G}^u(u_1,p_1) - \mathcal{G}^u(u_2, p_2)\|_{\IV}  
    &\leq L_{e, \IV}\|u_1-u_2\|_{\IV} + \|v_1-v
     _2\|_{\IV}.
\end{aligned}
\end{equation*}
By the formulation~\eqref{eq:Gp},
\begin{equation*}
    \begin{aligned}
    \|\mathcal{G}^p(u_1,p_1) - \mathcal{G}^p(u_2, p_2)\|_{\IQ}  \leq{} &  \|\nabla g_B(p_1)-\nabla g_B(p_2)\|_{\IQ^{-1}}  + \|B (e(u_1) - e(u_2)) \|_{\IQ^{-1}}  \\
\leq{}& L_{g_B, \IQ}\|p_1-p_2\|_{\IQ} + L_{e, \IV}L_S\|u_1-u_2\|_{\IV} 
\end{aligned}
\end{equation*}
where we have used
$$\lambda_{\max} \left(\IV^{-1}B^\top\IQ^{-1}B \right) = \lambda_{\max} \left(\IQ^{-1}B\IV^{-1}B^\top  \right) = L_S^2$$
to bound 
$$
\|B (e(u_1) - e(u_2)) \|_{\IQ^{-1}}^2 \leq L_S^2 \|e(u_1) - e(u_2) \|_{\IV}^2 \leq L_S^2 L_{e,\IV}^2\| u_1- u_2 \|_{\IV}^2. 
$$
\end{proof}
\correction{Notice that on the right hand side of \eqref{eq:LipGu}, $\|v_1-v _2\|_{\IV}$ appears which can be further bound by $\| u_1- u_2 \|_{\IV}$ and $\|p_1-p_2\|_{\IQ}$ using the triangle inequality. Here we keep $\|v_1-v _2\|_{\IV}$ with a neat Lipschitz constant $1$ and match the extra quadratic term in the strong Lyapunov property \eqref{eq: strong lyapunov in continuous level}.}

\section{Transformed Primal-Dual Iterations}\label{Discrete Schemes}
In this section, we derive several transformed primal-dual iterations, which are the discrete schemes for solving the TPD flow and obtain linear convergence rate based on the strong Lyapunov property. 


\subsection{Implicit Euler methods}
Given the initial guess $(u_0, p_0)$, for $k=0,1,\ldots,$ consider the implicit Euler method for the TPD flow~\eqref{eq: transformed primal-dual flow}:
\begin{equation}\label{IE discretization}
\left\{\begin{array}{l}\begin{aligned}
u_{k+1} &= u_k + \alpha_k \mathcal G^u(u_{k+1}, p_{k+1}), \\
p_{k+1}  &= p_k + \alpha_k \mathcal G^p(u_{k+1}, p_{k+1}).
\end{aligned}\end{array}\right.
\end{equation}

We show by the next theorem that the implicit scheme~\eqref{IE discretization} inherits the linear convergence rate from the strong Lyapunov property in the continuous level.

\begin{theorem}\label{convergence rate for IE discretization}
Suppose $f(u) \in \mathcal{S}_{\mu_{f,\IV}, L_{f,\IV} } $  with $0< \mu_{f,\IV} \leq L_{f, \IV}< 2$.  Let $(u_k, p_k)$ follows the implicit scheme~\eqref{IE discretization} for the TPD flow with initial value $(u_0, p_0)$, it holds that, for any $\alpha_k > 0$, 
$$\mathcal{E}(u_{k+1}, p_{k+1}) \leq \frac{1}{1+\alpha_k \mu}\mathcal{E}(u_k, p_k),\quad k\geq 0,$$
for the Lyapunov function defined by~\eqref{eq: Lyapunov} and $\mu =  \min \left \{\mu_{\mV},\mu_{\mQ}\right\}$.
\end{theorem}

\begin{proof}
Since $\mathcal{E}(u,p)$ is convex, we have
\begin{equation*}
    \begin{aligned}
    \mathcal{E}(u_{k+1}, p_{k+1}) - \mathcal{E}(u_k, p_k) 
\leq{}& \langle \nabla  \mathcal{E}(u_{k+1}, p_{k+1}), \alpha_k\mathcal{G}(u_{k+1}, p_{k+1})  \rangle  \\
\leq{} & -\alpha_k \mu  \mathcal{E}(u_{k+1}, p_{k+1}).
    \end{aligned}
\end{equation*}
The last inequality holds by the strong Lyapunov property~\eqref{eq: strong lyapunov in continuous level} in the continuous level. Then the linear convergence follows.
\end{proof}

For the implicit schemes, the larger the step size, the better the convergence rate. By increasing $\alpha_k$, the outer iteration may even achieve super-linear convergence. However,  the iteration~\eqref{IE discretization} is a nonlinear system with $u$ and $p$ coupled together. \correction{Consider the example when} $\IV = L_f I_m$ is a scaled identity and the proximal operator of $f$ is available, then we can solve $ u_{k+1}= {\rm prox}_{f, \alpha_k/L_f}(u_k - \frac{\alpha_k}{L_f}B^\top p_{k+1}) $ from the first equation of~\eqref{IE discretization} and substitute into the second to get a nonlinear equation of $p_{k+1}$
%
\begin{equation*}
    \begin{aligned}
    p_{k+1} &= p_k \correction{-} \IQ^{-1} \left [ \alpha_k \nabla g(p_{k+1}) + Bu_k - (1+\alpha_k)B\, {\rm prox}_{f, \frac{\alpha_k}{L_f}}\left (u_k - \frac{\alpha_k}{L_f}B^\top p_{k+1} \right) \right ]. \\
    \end{aligned}
\end{equation*}
If furthermore $\nabla {\rm prox}_{f, \alpha_k/L_f}$ is known, Newton's methods can be applied to solve this nonlinear equation. This is in the same \correction{spirit} of the semi-smooth Newton method developed in~\cite{Li;Sun;Toh:2018highly} for a non-smooth convex function $f$ (LASSO problem). 

\correction{In general, solving \eqref{IE discretization} may be as difficult as solving $\nabla \mathcal L(u,p) = 0$ and thus may not be practical. We shall explore more explicit schemes.}

%
%

\subsection{Explicit Euler methods}
An explicit discretization for~\eqref{eq: transformed primal-dual flow} is as follows:
\begin{equation}\label{EE discretization}
\left\{\begin{array}{l}\begin{aligned}
u_{k+1} &= u_k + \alpha_k \mathcal G^u(u_{k}, p_{k}), \\
p_{k+1}  &= p_k + \alpha_k \mathcal G^p(u_{k}, p_{k}).
\end{aligned}\end{array}\right.
\end{equation}
We present an equivalent but computationally favorable form of $\mathcal{G}^p(u,p)$
\begin{equation}
 \label{eq:GpUzawa}   
 \mathcal{G}^p(u,p) = - \IQ^{-1}\left [\nabla g(p) - B (u - \IV^{-1} (\nabla f(u)+ B^{\top}p ))  \right].
\end{equation}
Then~\eqref{EE discretization} is equivalent to
\begin{equation}
\left\{\begin{array}{l}
   \begin{aligned}\label{EE TPD}
u_{k+1/2} &= u_k - \IV^{-1}(\nabla f(u_k) + B^\top p_{k}),\\ 
p_{k+1}  &= p_k - \alpha_k \IQ^{-1}\left(\nabla g(p_k) - B u_{k+1/2} \right),\\
u_{k+1} &= (1-\alpha_k)u_k + \alpha_k u_{k+1/2}.
\end{aligned}
\end{array}\right. 
\end{equation}
The update of $(u_{k+1/2}, p_{k+1})$ is a variant of inexact Uzawa methods and $u_{k+1}$ is obtained by a weighted average of $u_k$ and $u_{k+1/2}$. The convergence is clear in the formulation~\eqref{EE discretization}. 


\begin{theorem}\label{convergence rate for EE discretization}
Suppose $f(u) \in \mathcal{S}_{\mu_{f,\IV}, L_{f,\IV} }$  with $0< \mu_{f,\IV} \leq L_{f, \IV}< 2$. Let $(u_k, p_k)$ follows the explicit scheme ~\eqref{EE discretization} for the TPD flow with initial value $(u_0, p_0)$. For the Lyapunov function defined by~\eqref{eq: Lyapunov}, it holds that
$$\mathcal{E}(u_{k+1}, p_{k+1}) \leq (1-\delta_k)\mathcal{E}(u_k, p_k)$$
for $0 < \alpha_k < \displaystyle \min \left \{\mu_{\mV}/L_{\mV}^2, \mu_{\mQ}/L_{\mQ}^2, \mu_{f,\IV}/2 \right \}$
and $$0 < \delta_k = \min \left \{ \alpha_k(\mu_{\mV} - L_{\mV}^2\alpha_k), \alpha_k\left ( \mu_{\mQ}- L_{\mQ}^2\alpha_k\right) \right\}<1.$$ 
In particular, for $\alpha_k = 0.5\min \{\mu_{\mV}, \mu_{\mQ}\}/\max \{ L_{\mV}^2, L_{\mQ}^2, 2\}$, we have the linear rate
$$\mathcal{E}(u_{k+1}, p_{k+1}) \leq (1-\frac{1}{4\kappa^2})\mathcal{E}(u_k, p_k),$$
with $\kappa \geq \max\{ \kappa_{\mV}, \kappa_{\mQ}\} , \kappa_{\mV} := \max \{L_{\mV}, 2\}/\mu_{\mV}, \kappa_{\mQ} := L_{\mQ}/\mu_{\mQ}$. 
\end{theorem}
\begin{proof}
Since $\mathcal{E}(u,p)$ is quadratic and convex, we have
\begin{equation}\label{convexity of E(u,p)}
    \begin{aligned}
    \mathcal{E}(u_{k+1}, p_{k+1}) - \mathcal{E}(u_k, p_k) ={} & \langle \partial_u  \mathcal{E}(u_k, p_k) , u_{k+1} - u_k \rangle + \frac{1}{2} \|u_{k+1}-u_k\|^2_{\IV} \\
    &+ \langle \partial_p  \mathcal{E}(u_k, p_k) , p_{k+1} - p_k \rangle + \frac{1}{2} \|p_{k+1}-p_{k}\|^2_{\IQ}.    
    \end{aligned}
\end{equation}
By formulation~\eqref{EE discretization} and the strong Lyapunov property established in Theorem \ref{Continuous strong Lyapunov for transformed gradient flow},
\begin{equation}\label{continuous level result}
    \begin{aligned}
   &\langle \partial_v  \mathcal{E}(u_k, p_k) , u_{k+1} - u_k \rangle + \langle \partial_p  \mathcal{E}(u_k, p_k) , p_{k+1} - p_k \rangle \\
 ={}  & \langle \nabla \mathcal{E}(u_k, p_k), \alpha_k\mathcal{G}(u_k, p_k)  \rangle \\
\leq {}&-\frac{\alpha_k\mu_{\mV}}{2} \|u_k-u^*\|^2_{\IV} - \frac{\alpha_k \mu_{\mQ}}{2} \|p_k-p^*\|^2_{\IQ} - \frac{\alpha_k \mu_{f,\IV}}{2}\| v_k - v^* \|_{\IV}^2.
    \end{aligned}
\end{equation}

By the Lipschitz continuity of the flow, cf. Lemma \ref{Lipschitz continuity of flow},
\begin{equation}\label{quadratic terms}
    \begin{aligned}
   &\frac{1}{2}\|u_{k+1}-u_k\|^2_{\IV}+ \frac{1}{2} \|p_{k+1}-p_{k}\|^2_{\IQ}\\
   ={}&\frac{\alpha_k^2}{2}\left (\|\mathcal{G}^u(u_k,p_k) - \mathcal{G}^u(u^*,p^*)\|^2_{\IV} + \|\mathcal{G}^p(u_k,p_k) - \mathcal{G}^p(u^*,p^*)\|^2_{\IQ}\right ) \\
    \leq{}& \frac{\alpha_k^2L_{\mV}^2}{2} \|u_k-u^*\|^2_{\IV} + \frac{\alpha_k^2L_{\mQ}^2}{2}\|p_k-p^*\|^2_{\IQ} + \alpha_k^2 \| v_k - v^*\|^2.
    \end{aligned}
\end{equation}
Summing~\eqref{continuous level result} and~\eqref{quadratic terms},
\begin{equation*}
    \begin{aligned}
    \mathcal{E}(u_{k+1}, p_{k+1}) - \mathcal{E}(u_k, p_k) \leq &  
 -\alpha_k\left (\mu_{\mV} - \alpha_k L_{\mV}^2 \right) \frac{1}{2} \|u_k-u^*\|^2_{\IV} \\
    &- \alpha_k\left (\mu_{\mQ} - \alpha_k L_{\mQ}^2 \right) \frac{1}{2}\|p_k-p^*\|^2_{\IQ}\\
   &- \alpha_k (\mu_{f,\IV}/2 - \alpha_k )  \| v_k - v^*\|^2.
\end{aligned}
\end{equation*}
Then the results follows by rearrangement of the inequality and bound of the quadratic polynomial of $\alpha_k$. 
\end{proof}

We can always rescale the function $f$ or $\IV$ so that $L_{f, \IV} \leq 1$ and consequently $L_{e,\IV} < 1$. We can also rescale $\IQ$ so that $\lambda_{\max} \left (\IQ^{-1}B\IV^{-1}B^\top \right ) \leq 1$. Consequently $L_{\mV}^2 \leq 4$ and $L_{\mQ}^2 = O(L^2_{g,\IQ} + 1)$. Theorem \ref{convergence rate for EE discretization} shows the convergence rate is determined by the condition number $\kappa_{\mV} = O(\kappa_{f,\IV})$ and $\kappa_{\mQ} =  O(\kappa(\IQ^{-1}B\IV^{-1}B^\top ))$ which in turn depends crucially on choices of $\IV$ and $\IQ$. 

Both $\IV$ and $\IQ$ can be scalars, then~\eqref{eq:GpUzawa} is an explicit first order method with linear convergence rate. However, in this case, when either $\kappa(f)$ or $\kappa(BB^{\top})$ is large, the convergence will be very slow \correction{since the rate is degenerate like $1-c/\kappa^2$}.

We can choose an SPD matrix $\IV$ to make $f$ better conditioned.  
%
%
As $g$ is convex only, i.e., $\mu_g$ might be zero, the convexity $\mu_{\mQ}\geq \lambda_{\min} \left (\IQ^{-1}B\IV^{-1}B^\top \right )$. 
In the ideal case, we choose $\IQ^{-1} = (B\IV^{-1}B^\top)^{-1}$ and then $\mu_{\mQ} = 1 + \mu_g$ but in practice $ (B\IV^{-1}B^\top)^{-1}$ may not be able to be computed efficiently. When $\IV^{-1} = A^{-1}$ is dense, even the Schur complement $B\IV^{-1}B^\top$ may not be formed explicitly. Without a priori information on the Schur complement, it is hard to choose $\IQ$ to make $\kappa_{\mQ}$ small. A scalar $\IQ$ will lead to $\kappa_{\mQ} = \kappa( B\IV^{-1}B^\top)$ which competes with $\kappa_{f,\IV}$.

\correction{After choosing $\IV$ and $\IQ$, the optimal step size is the $\alpha_k$ that reaching the upper bound of quadratic functions to determine $\delta_k$. If the convexity constants $\mu$'s and the Lipschitz constants of gradients $L$'s are given (or can be estimated), then Theorem \ref{convergence rate for EE discretization} gives analytical guidance for choosing the step size.  In practice, one can start from $\alpha_k=1$ and decrease the step size with a fixed ratio, e.g. $1/2$, until the residual is reduced.}

\subsection{Implicit-Explicit Methods}\label{sec: IMEX scheme}
For the explicit scheme, the step size should be small enough and the convergence rate is $1-c/\kappa^2$ which is very slow if either $\kappa_{\mV}$ or $\kappa_{\mQ}$ is large. Can we \correction{enlarge the step size and }accelerate this linear rate? 

\correction{One way} is to apply the Implicit-Explicit (IMEX) scheme for solving the TPD flow~\eqref{eq: transformed primal-dual flow}. Given an initial $(u_0, p_0)$, for $k=0,1,\ldots ,$ update $(u_{k+1}, p_{k+1})$ as follows:
\begin{equation}\label{IMEX discretization}
\left\{\begin{array}{l}\begin{aligned}
p_{k+1}  &= p_k +\alpha_k \mathcal G^p(u_k, p_k), \\
u_{k+1} &= u_k+\alpha_k  \mathcal G^u(u_{k+1}, p_{k+1}). 
\end{aligned}\end{array}\right.
\end{equation}
\correction{That is,} we update $p$ by the explicit Euler method and solve $u$ by the implicit Euler method. 
Again we can view~\eqref{IMEX discretization} as a correction to the inexact Uzawa method
\begin{equation}\label{IMEX algorithm}
\left\{\begin{array}{l}\begin{aligned}
u_{k+1/2} &= u_k - \IV^{-1}(\nabla f(u_k) + B^\top p_{k}),\\ 
p_{k+1}  &= p_k - \alpha_k \IQ^{-1}\left(\nabla g(p_k) - B u_{k+1/2} \right), \\
u_{k+1} &= \arg \min_{u \in \mV} f(u) + \frac{1}{2\alpha_k} \|u - u_k + \alpha_k\IV^{-1}B^\top p_{k+1}\|^2_{\IV}. 
\end{aligned}\end{array}\right.
\end{equation}
After one inexact Uzawa iteration, $u_{k+1}$ is obtained by solving a strongly convex optimization problem of $u$. When $\IV = L_f I_m$, the last step is one proximal iteration
$$
u_{k+1} = {\rm prox}_{ f, \frac{\alpha_k}{L_f}}  \left (u_k - \frac{\alpha_k}{L_f} B^\top p_{k+1}\right).
$$
%

\correction{We can also use IMEX schemes with updating $u$ first with proximal iteration and $p$ later using $u_{k+1}-u_k$. Specific $\IQ = \frac{1}{r}BB^{\top} + \delta I$ is discussed in~\cite{he2021balanced} where $\IV = r I$ with arbitrary $r >0$  and step size $\alpha_k =1$ is allowed. Our analysis is unified for general $\IV$ and $\IQ$ using the Lyapunov function.} Compared with the explicit scheme, the IMEX scheme enjoys accelerated linear convergence rates.
\begin{theorem}\label{convergence rate for IMEX discretization}
Suppose $f(u) \in \mathcal{S}_{\mu_{f,\IV}, L_{f,\IV} }$  with $0< \mu_{f,\IV} \leq L_{f, \IV}< 2$. Let $(u_k, p_k)$ follows the IMEX scheme~\eqref{IMEX algorithm} for the TPD flow with initial value $(u_0, p_0)$. For the Lyapunov function defined by~\eqref{eq: Lyapunov}, it holds that
$$\mathcal{E}(u_{k+1}, p_{k+1}) \leq \frac{1}{1+\alpha_k \mu_k}\mathcal{E}(u_k, p_k),$$
for $0 < \alpha_k < \displaystyle \mu_{\mQ}/L_{S,\mQ}^2$ and $\mu_k = \min \left \{\mu_{\mV} , \mu_{\mQ} - \alpha_kL_{S,\mQ}^2\right\}.$
In particular, for $\alpha_k = 0.5 \mu_{\mQ}/ L_{S,\mQ}^2$, we have 
$$\mathcal{E}(u_{k+1}, p_{k+1}) \leq \frac{1}{1+ 0.5 \mu_{\mQ} \min \{\mu_{\mV}, \mu_{\mQ}/2\} / L_{S,\mQ}^2 }\mathcal{E}(u_k, p_k).$$
\end{theorem}
\begin{proof}
Since $\mathcal{E}(u,p)$ is quadratic and convex, we have
\begin{equation}\label{eq: convexity of E(u,p) at k+1}
\begin{aligned}
    &\mathcal{E}(u_{k+1}, p_{k+1}) - \mathcal{E}(u_k, p_k) \\
    ={} & \langle \partial_u  \mathcal{E}(u_{k+1}, p_{k+1}) , u_{k+1} - u_k \rangle - \frac{1}{2} \|u_{k+1}-u_k\|^2_{\IV} \\
    &+ \langle \partial_p  \mathcal{E}(u_{k+1}, p_{k+1}) , p_{k+1} - p_k \rangle - \frac{1}{2} \|p_{k+1}-p_{k}\|^2_{\IQ}.    
    \end{aligned}
\end{equation}

We will use the strong Lyapunov property at $(u_{k+1}, p_{k+1})$ but the component $\mathcal G^p(u_k,p_k)$ is evaluated at $(u_k, p_k)$. Compared with the implicit scheme, there are some mis-match terms from the explicit step for $p$:
\begin{equation}\label{eq: continuous level result 2}
    \begin{aligned}
   &\langle \partial_u  \mathcal{E}(u_{k+1}, p_{k+1}) , u_{k+1} - u_k \rangle + \langle \partial_p  \mathcal{E}(u_{k+1}, p_{k+1}) , p_{k+1} - p_k \rangle \\
   ={} & \langle \nabla  \mathcal{E}(u_{k+1}, p_{k+1}), \alpha_k\mathcal{G}(u_{k+1}, p_{k+1})  \rangle \\
   & +  \alpha_k \langle p_{k+1}-p^*,   \nabla g_B(p_{k+1})- \nabla  g_B(p_k) + B \left (e(u_k) - e(u_{k+1} \right) \rangle\\
 \leq &-\frac{\alpha_k\mu_{\mV}}{2} \|u_{k+1}-u^*\|^2_{\IV} -\frac{\alpha_k \mu_{\mQ}}{2} \|p_{k+1}-p^*\|^2_{\IQ} \\
   & +  \alpha_k \langle p_{k+1}-p^*,   \nabla g_B(p_{k+1})- \nabla  g_B(p_k) + B \left (e(u_k) - e(u_{k+1} \right) \rangle.
    \end{aligned}
\end{equation}
We use \correction{Cauchy-Schwarz} inequality to bound the mis-match terms in~\eqref{eq: continuous level result 2}:
$$
    \begin{aligned}
    & \alpha_k \langle p_{k+1}-p^*,   \nabla g_B(p_{k+1})- \nabla  g_B(p_k) + B \left (e(u_k) - e(u_{k+1}) \right) \rangle \\
     \leq{} & \frac{\alpha_k^2}{2} \left (L_{e,\IV}^2L_S^2 + L^2_{g_B, \IQ} \right) \| p_{k+1}-p^*\|^2_{\IQ} + \frac{1}{2L_{g_B, \IQ}^2}\| \nabla g_B(p_{k+1})- \nabla  g_B(p_k)\|_{\IQ^{-1}}^2\\
     &+ \frac{1}{2L_{e,\IV}^2L_S^2} \|B \left (e(u_{k+1}) - e(u_k)\right)\|^2_{\IQ^{-1}} \\
    \leq{} & \frac{\alpha_k^2}{2}L_{S,\mQ}^2 \| p_{k+1}-p^*\|^2_{\IQ} +  \frac{1}{2} \|p_{k+1} - p_k\|^2_{\IQ} +\frac{1}{2} \|u_{k+1} - u_k\|^2_{\IV}.
    \end{aligned}
$$    
Use the negative terms in~\eqref{eq: convexity of E(u,p) at k+1}, we obtain
\begin{equation*}
    \begin{aligned}
    &\mathcal{E}(u_{k+1}, p_{k+1}) - \mathcal{E}(u_k, p_k) \\
    \leq &   -\frac{\alpha_k\mu_{\mV}}{2} \|u_{k+1}-u^*\|^2_{\IV} -\frac{1}{2}\alpha_k \left ( \mu_{\mQ} - \alpha_kL_{S,\mQ}^2 \right) \|p_{k+1} - p^*\|^2_{\IQ}.
\end{aligned}
\end{equation*}
Then the results follows by rearrangement of the inequality and bound of the quadratic polynomial of $\alpha_k$. 

%
\end{proof}

Let us discuss the rate with assumption $\lambda_{\max} \left (\IQ^{-1}B\IV^{-1}B^\top \right ) \leq 1$ and $\mu_{\mV}\leq  \mu_{\mQ}/2$. Theorem \ref{convergence rate for IMEX discretization} shows the convergence rate of the IMEX scheme is $\left( 1+ c\mu_{\mQ}\mu_{\mV}\right )^{-1}$. When both $\mu_{\mQ}$ and $\mu_{\mV}$ are small, the linear rate is still in the quadratic dependence of condition numbers. The improvement is that if we can choose $\IQ$ such that \correction{$\mu_{\mQ}\gg \mu_{\mV}$}, then we achieve the accelerated rate $(1+ c/\kappa_{\mV})^{-1}$. While for the explicit scheme, even $\kappa_{\mQ}$ is small, the rate is still \correction{worse than $1- c/\max^2\{\kappa_{\mV},\kappa_{\mQ}\} = 1-c/\kappa_{\mV}^2$}. 

Augmented Lagrangian can be viewed as a preconditioning of the Schur complement so that a simple $\IQ^{-1} = \beta I_n$ will lead to a well conditioned $\kappa_{\mQ}$; see Section \ref{sec:ALM} for details.

\correction{The largest step size $\alpha_k$ is still in the order of $\mu_{\mQ}$. As $u$ is treat implicitly, there is no restriction of the step size from $\mu_{\mV}$. In Section \ref{sec:GS} we shall propose an explicit method with enlarged step size and accelerated convergence rate.}

\subsection{Inexact inner solvers}\label{Inexact inner solvers} For those TPD iterations, the most time consuming part is the inner solver for sub-problems. For the explicit scheme~\eqref{EE discretization}, that is the linear operators $\IV^{-1}$ and $\IQ^{-1}$. For example, when $\IV = L_f I$, if we treat $ L_f(BB^{\top})^{-1}$ as the ideal exact inner solve, then $\kappa_{\mQ} = 1$. A general $\IQ^{-1}$ can be \correction{treated} as an inexact inner solver and the inexactness enters the estimate by $\lambda_{\min}\left (\IQ^{-1}B\IV^{-1}B^\top \right )$.  

 \correction{For the IMEX scheme, the sub-problem  in the third step of~\eqref{IMEX algorithm} is a strongly convex optimization problem.} In this part, we derive the perturbation analysis for inexact inner solvers for this sub-problem. 


Define the modified objective function for this sub-problem
\begin{equation}\label{inner solve problem}
    \tilde f(u; u_k, p_{k+1}) =  f(u) + \frac{1}{2\alpha_k} \|u - u_k + \alpha_k \IV^{-1}B^\top p_{k+1}\|^2_{\IV}, 
\end{equation}
the inexactness of the inner solve is measured by $\|\nabla \tilde f(u)\|^2.$

\begin{theorem}\label{thm: convergence of inexact IMEX TPD}
Suppose $f(u) \in \mathcal{S}_{\mu_{f,\IV}, L_{f,\IV} } $  with $0< \mu_{f,\IV} \leq L_{f, \IV}< 2$. Suppose $(u_k, p_k)$ follows the inexact IMEX iteration~\eqref{IMEX algorithm} with initial value $(u_0, p_0)$ and the inexact inner solver returns $u_{k+1}$
satisfying $\|\nabla \tilde f(u_{k+1})\|^2_{\IV^{-1}} \leq \epsilon_k$ for $k = 1,2, \cdots$. Then for the Lyapunov function defined by~\eqref{eq: Lyapunov}, it holds that
$$\mathcal{E}(u_{k+1}, p_{k+1}) \leq \frac{1}{1+\alpha_k \mu_k}\mathcal{E}(u_k, p_k)+ \frac{\correction{\alpha_k}}{(1+\alpha_k \mu_k)\mu_{\mV}}\epsilon_k ,$$
for $0 < \alpha_k < \displaystyle \mu_{\mQ}/L_{S,\mQ}^2$ and $\mu_k = \min \left \{\mu_\mV/2, \mu_{\mQ} - \alpha_kL_{S,\mQ}^2 \right \}$.
In particular, for $\alpha_k = \mu_{\mQ}/ 2L_{S,\mQ}^2$, the accumulative perturbation error for the inexact solve is
$$\mathcal{E}(u_{n+1}, p_{n+1}) \leq \rho^{n+1}\mathcal{E}(u_0, p_0)+ \frac{\mu_{\mQ}}{\correction{2}\mu_{\mV}L_{S,\mQ}^2}\sum_{k=0}^n \rho^{n-k+1} \epsilon_k ,$$
where $\mu = \min\{\mu_\mV, \mu_\mQ \}$ and $\rho =1/(1+ \mu_{\mQ} \mu/ 4L_{S,\mQ}^2)\in (0,1)$.

\end{theorem}

\begin{proof}
By definition~\eqref{inner solve problem},
\begin{equation*}
\nabla \tilde f(u_{k+1}) = \nabla f(u_{k+1}) + \frac{1}{\alpha_k} \left (\IV u_{k+1} - \IV u_k +\alpha_k B^\top p_{k+1} \right),
\end{equation*}
we can write
\begin{equation*}
\begin{aligned}
u_{k+1} - u_k  &= \alpha_k \IV^{-1} \left ( \nabla \tilde f(u_{k+1}) - \nabla f(u_{k+1})-  B^\top p_{k+1} \right )  \\
& = \alpha_k \left ( \IV^{-1} \nabla \tilde f(u_{k+1}) + \mathcal G^u (u_{k+1}, p_{k+1}) \right ).
\end{aligned}
\end{equation*}

We use the strong Lyapunov property at $(u_{k+1}, p_{k+1})$ but compared with~\eqref{eq: continuous level result 2} , we have an additional gradient term due to the inexact inner solve:
\begin{equation*}
    \begin{aligned}
    &\mathcal{E}( u_{k+1}, p_{k+1}) - \mathcal{E}(u_k, p_k) \\
    ={} & \langle \partial_u  \mathcal{E}(u_{k+1}, p_{k+1}) , u_{k+1} -  u_k \rangle - \frac{1}{2} \| u_{k+1}- u_k\|^2_{\IV} \\
    &+ \langle \partial_p  \mathcal{E}( u_{k+1}, p_{k+1}) , p_{k+1} - p_k \rangle - \frac{1}{2} \|p_{k+1}-p_{k}\|^2_{\IQ}  \\
    \leq{} & \langle \partial_u  \mathcal{E}(u_{k+1}, p_{k+1}) , \alpha_k \mathcal G^u(u_{k+1}, p_{k+1}) \rangle+ \langle \partial_p  \mathcal{E}( u_{k+1}, p_{k+1}) , \alpha_k \mathcal G^p(u_{k}, p_{k}) \rangle  \\
    &- \frac{1}{2} \| u_{k+1}- u_k\|^2_{\IV} - \frac{1}{2} \|p_{k+1}-p_{k}\|^2_{\IQ}+ \langle \partial_u  \mathcal{E}(u_{k+1}, p_{k+1}) , \alpha_k \IV^{-1}\nabla \tilde f(u_{k+1}) \rangle  \\
    \leq{}& -\frac{\alpha_k\mu_{\mV}}{4} \|u_{k+1}-u^*\|^2_{\IV} -\frac{1}{2}\alpha_k \left ( \mu_{\mQ} - \alpha_kL_{S,\mQ}^2 \right) \|p_{k+1} - p^*\|^2_{\IV} + \correction{\frac{\alpha_k}{\mu_{\mV}}} \| \nabla\tilde f(u_{k+1})\|_{\IV^{-1}}^2,
    \end{aligned}
\end{equation*}
where the last inequality holds from Theorem \ref{convergence rate for IMEX discretization} and by Cauchy-Schwarz inequality
\begin{equation*}
    \begin{aligned}
    \langle \partial_u  \mathcal{E}(u_{k+1}, p_{k+1}) , \alpha_k \IV^{-1}\nabla \tilde f(u_{k+1}) \rangle &= \langle  \IV \left (u_{k+1} - u^* \right ), \alpha_k \IV^{-1}\nabla \tilde f(u_{k+1}) \rangle\\
    &\leq \frac{\alpha_k\mu_{\mV}}{4} \|u_{k+1}-u^*\|^2_{\IV} + \correction{\frac{\alpha_k}{\mu_{\mV}}} \|\nabla \tilde f(u_{k+1})\|_{\IV^{-1}}^2.
    \end{aligned}
\end{equation*}
Since the inexact solver terminates until $\|\nabla \tilde f(u_{k+1})\|^2_{\IV^{-1}} < \epsilon_k$, we have
\begin{equation*}
    \begin{aligned}
    \mathcal{E}( u_{k+1}, p_{k+1}) - \mathcal{E}(u_k, p_k) \leq -\alpha_k \mu_k \mathcal{E}( u_{k+1}, p_{k+1}) + \frac{\alpha_k\epsilon_k}{\mu_{\mV}} 
    \end{aligned}
\end{equation*}
with $\mu_k = \min \left \{\mu_\mV/2, \mu_{\mQ} - \alpha_kL_{S,\mQ}^2 \right \}$ and the accumulated error is straight forward.
\end{proof}

For $\alpha = \alpha_k = \mu_{\mQ}/ 2L_{S,\mQ}^2$ and $\correction{\epsilon_k \leq  \mu \mu_{\mV}\epsilon}$ for some $\epsilon > 0$, the accumulated perturbation error 
$$\frac{\mu_{\mQ}}{\correction{2}\mu_{\mV}L_{S,\mQ}^2}\sum_{k=0}^n \rho^{n-k+1} \epsilon_k \leq \alpha \mu \epsilon\sum_{k=0}^n \left(\frac{1}{1+\alpha\mu}\right)^{k+1} \leq \epsilon. $$
 Furthermore, in the product $\rho^{n-k+1} \epsilon_k$, the weight $\rho^{n-k+1}$ is geometrically increasing, we can choose relative large $\epsilon_k$ in the beginning and gradually decrease $\epsilon_k$. On the other hand, when the outer iteration converges, the initial guess $u_k$ for the sub-problem
\begin{equation*}
\begin{aligned}
\nabla \tilde f(u_{k}) = \nabla f(u_{k}) + B^\top p_{k+1} = \partial_u \mathcal L(u_{k},p_{k})+ B^\top (p_{k+1} - p_k) \to 0
\end{aligned}
\end{equation*}
is already small. A smaller $\epsilon_k$ can be achieved for constant inner iteration steps. Therefore the inexact IMEX scheme retains the accelerated linear convergence rates.


\subsection{\correction{A Gauss-Seidel iteration with accelerated overrelaxation}}\label{sec:GS}

In this subsection, we propose an explicit scheme for the transformed primal-dual flow: a Gauss-Seidel iteration with accelerated overrelaxation (AOR)~\cite{hadjidimos1978accelerated}: 
\begin{equation}\label{eq:G-S TPD}
\left \{\begin{aligned}
  \frac{u_{k+1} - u_k}{\alpha}&= -\IV^{-1}(\nabla f(u_{k}) + B^{\top}p_{k}) \\
\frac{p_{k+1} - p_k}{\alpha} &=-\IQ^{-1}\left [ B\IV^{-1}\nabla f(u_{k+1}) + \nabla g_B(p_{k})- B(2u_{k+1} - u_k)\right ].      
\end{aligned} \right.
\end{equation}
The formulation \eqref{eq:G-S TPD} is in Gauss-Seidel type as when updating $p_{k+1}$, the updated $u_{k+1}$ is used. AOR is applied to the term $Bu\approx B(2u_{k+1} - u_k)$ with an overrelaxation parameter $2$. Such change is motivated by accelerated overrelaxtion methods \cite{hadjidimos1978accelerated} and the linear convergence rate is indeed accelerated to $(1+c/\kappa)^{-1}$.

For a symmetric matrix $M$, we define $$\|x\|_M^2 :=(x,x)_M := x^{\top}Mx.$$ When $M$ is SPD, it defines an inner product and the induced norm. For a general symmetric matrix, $\|\cdot \|_M$ may not be a norm. However the following identity for squares still holds
\begin{equation}\label{eq:squares}
	2(a,b)_M = \|a\|_M^2 + \|b\|_M^2 - \|a-b\|_M^2.
\end{equation}
Let $ \mathcal M_{\mathcal X} = \text{diag}\{\IV, \IQ\}$ and $x = (u, p)$. Then we have 
\begin{equation*}
    \frac{1}{2}\|x -x ^*\|_{\mathcal M_{\mathcal X} }^2  =  \frac{1}{2}\|u -u ^*\|_{\IV}^2 +\frac{1}{2}\|p -p ^*\|_{\IQ}^2. 
\end{equation*}
Now we are ready to prove the convergence rate. Consider the Lyapunov function 
 \begin{equation}\label{eq: modified Lyapunov}
    \mathcal{E} (x)= \frac{1}{2}\|x-x^{
    \star}\|^2_{\mathcal M_{\mathcal X}- \alpha\mathcal B}  -\alpha D_f(u^*, u) - \alpha D_{g_B}(p^*, p).
\end{equation}
where recall that $\mathcal B = \begin{pmatrix}
       0&  B^{\top}\\
       B& 0
\end{pmatrix}$ is a symmetric matrix and $D_f$ and $D_{g_B}$ are Bregman divergence of $f$ and $g_B$, respectively. 

\begin{lemma}\label{lem: positivity of Lyapunov}
For $\alpha  < 1/\max \{ 2L_S, 2L_{f,\IV}, 2L_{g_B, \IQ} \}$ , for the Lyapunov function $\mathcal E$ defined by \eqref{eq: modified Lyapunov}, we have  $\mathcal E(x)\geq 0$ and $\mathcal E(x)= 0$  if and only if $x = x^*$. 
\end{lemma}

\begin{proof}
Notice \begin{equation}\label{eq:Mx-B}
\begin{aligned}
         \mathcal M_{\mathcal X}- 2\alpha\mathcal B &=  \begin{pmatrix} \IV & -2\alpha B^{\top} \\
    -2\alpha B & \IQ
   \end{pmatrix} \\
   &=\begin{pmatrix} \mathcal{I} & 0 \\
    -2\alpha B\IV^{-1} & \mathcal{I}
   \end{pmatrix}\begin{pmatrix} \IV & 0 \\
    0 & \IQ- 4\alpha^2B\IV^{-1}B^{\top}
   \end{pmatrix} \begin{pmatrix}\mathcal{I} &  -2\alpha \IV^{-1} B^{\top} \\
  0  & \mathcal{I}
   \end{pmatrix}.
\end{aligned}
\end{equation}
We have
\begin{equation}\label{eq: positivity bound 1}
    \frac{1}{2}\|x-x^*\|_{ \frac{1}{2}\mathcal M_{\mathcal X}- \alpha\mathcal B}^2 =      \frac{1}{4}\|x-x^*\|_{ \mathcal M_{\mathcal X}- 2\alpha\mathcal B}^2 = \frac{1}{4}\|y-y^*\|^2_{\mathcal M_{\mathcal Y}} \geq 0
\end{equation}
where the change of variable is $y =  \begin{pmatrix}\mathcal{I} &  -2\alpha \IV^{-1} B^{\top} \\
  0  & \mathcal{I}
   \end{pmatrix} x$ and $$\mathcal M_{\mathcal Y} = \begin{pmatrix} \IV & 0 \\
    0 & \IQ- 4\alpha^2B\IV^{-1}B^{\top}
   \end{pmatrix}$$ is positive definite if $\alpha < 1/(2L_S)$. In particular, the equality is obtained if and only if $y= y^*$, which is equivalent to $x=x^*$ since the change of coordinate is invertible.
   
 For $\alpha < 1/\max \{2L_{f,\IV}, 2L_{g_B, \IQ}\}$, we have 
 \begin{equation}\label{eq: positivity bound 2}
 \begin{aligned}
        \frac{1}{2}\|x -x^*\|_{\frac{1}{2}\mathcal M_{\mathcal X}}^2 &= \frac{1}{4}\|u-u^*\|_{\IV}^2 +   \frac{1}{4}\|p-p^*\|_{\IQ}^2\\
        &\geq \frac{1}{2L_{f,\IV}}D_f(u^*, u) +  \frac{1}{2L_{g_B,\IQ}}D_{g_B}(p^*, p) \\
        &\geq \alpha D_f(u^*, u) +  \alpha D_{g_B}(p^*, p).
 \end{aligned}
 \end{equation}
The last inequality becomes equality if and only if $D_f(u^*, u) = D_{g_B}(p^*, p) = 0$, which is equivalent to $u = u^*, p=p^*$.

Sum \eqref{eq: positivity bound 1} and \eqref{eq: positivity bound 2} we get the desired inequality
\begin{equation*}
    \mathcal E(x) = \frac{1}{2}\|x-x^*\|_{\mathcal M_{\mathcal X}- \alpha\mathcal B}^2 - \alpha D_f(u^*, u) - \alpha D_{g_B}(p^*, p) \geq 0
\end{equation*}
for $\alpha  < 1/\max \{ 2L_S, 2L_{f,\IV}, 2L_{g_B, \IQ} \}$ and the equality holds is and only if $x = x^*$.
\end{proof}

Then we show the accelerated linear convergence rate.

\begin{theorem}\label{thm: convergence of G-S TPD}
Suppose $f(u) \in \mathcal{S}_{\mu_{f,\IV}, L_{f,\IV} }$  with $0< \mu_{f,\IV}\leq L_{f, \IV}< 2$. Let $x_k=(u_k,p_k)$ be generated by GS-AOR iteration \eqref{eq:G-S TPD} with initial value $x_0=(u_0, p_0)$ and $\alpha < 1/\max \{ 2L_S, 2L_{f,\IV}, 2L_{g_B, \IQ} \}$.  Then for the discrete Lyapunov function \eqref{eq: modified Lyapunov}, we have 
\begin{equation}
\begin{aligned}
\mathcal{E}(x_{k+1})\leq&~ \frac{1}{1+\mu \alpha/2}\mathcal{E}(x_k).
\end{aligned}
\end{equation}
where $\mu = \min\left \{\mu_{\mV},\mu_{\mQ} \right \}$.
\end{theorem}
\begin{proof}
We use the identity for squares \eqref{eq:squares}:
\begin{equation}\label{eq:Eqdiff}
\begin{aligned}
  \frac{1 }{2} \| x_{k+1} - x^*\|_{\mathcal M_{\mathcal X}}^2-\frac{1 }{2} \| x_{k} - x^*\|_{\mathcal M_{\mathcal X}}^2=  \langle  x_{k+1} -x^*, x_{k+1} - x_k \rangle_{\mathcal M_{\mathcal X}} -\frac{1}{2} \| x_{k+1} - x_k\|_{\mathcal M_{\mathcal X}}^2. 
\end{aligned}
\end{equation}

We write the scheme \eqref{eq:G-S TPD} as a correction of the implicit Euler scheme
\begin{align*}
	u_{k+1} - u_k & = \alpha (\mathcal G^u ( x_{k+1}) -\mathcal G^u(x^*))+ \alpha \IV^{-1} B^{\top}(p_{k+1} - p_k) + \alpha \IV^{-1}(\nabla f(u_{k+1}) - \nabla f(u_k)), \\
	p_{k+1} - p_k & = \alpha(\mathcal G^p ( x_{k+1}) -\mathcal G^p(x^*))  + \alpha \IQ^{-1}B(u_{k+1} - u_k)+ \alpha \IQ^{-1}(\nabla g_B(p_{k+1}) - \nabla g_B(p_k)).
\end{align*}

Recall that, for the TPD flow, we have proved in Theorem \ref{Continuous strong Lyapunov for transformed gradient flow} that
$$
\langle \mathcal M_{\mathcal X} (x_{k+1}-x^*),  \mathcal G ( x_{k+1}) -\mathcal G ( x^*)  \rangle \leq - \frac{\mu}{2}\|x_{k+1} -x^*\|^2_{\mathcal M_{\mathcal X}}. 
$$
We merge the first cross terms and use the identity \eqref{eq:squares} to expand as
\begin{align*}
  &(  u_{k+1} -u^*, B^{\top} (p_{k+1} - p_k)) + (  p_{k+1} -p^*, B(u_{k+1} - u_k)) \\
  ={} & (  x_{k+1} -x^*, x_{k+1} - x_k)_{\mathcal B} \\
  ={} & \frac{1}{2}(\| x_{k+1} -x^* \|_{\mathcal B}^2 +  \| x_{k+1} -x_k \|_{\mathcal B}^2 - \| x_{k} -x^* \|_{\mathcal B}^2).
\end{align*}
The other cross terms with the Bregman divergence is expanded using the identity \eqref{eq: Bregman divergence identity}
\begin{equation*}
\begin{aligned}
   \langle  u_{k+1} - u^*, \nabla f(u_{k+1}) - \nabla f(u_k) \rangle & = D_f(u^*, u_{k+1}) + D_f( u_{k+1}, u_k) - D_f( u^*, u_k), \\
  \langle  p_{k+1} - p^*,\nabla g_B(p_{k+1}) - \nabla g_B(p_k) \rangle & = D_{g_B}(p^*, p_{k+1}) + D_{g_B}( p_{k+1}, p_k) - D_{g_B}( p^*, p_k) .
\end{aligned}
\end{equation*}
Substituting back to \eqref{eq:Eqdiff} we obtain the inequality
\begin{equation*}
\begin{aligned}
&  \frac{1 }{2} \| x_{k+1} - x^*\|_{\mathcal M_{\mathcal X}}^2-\frac{1 }{2} \| x_{k} - x^*\|_{\mathcal M_{\mathcal X}}^2 \\
\leq{}  &- \frac{\mu \alpha}{2}\|x_{k+1} -x^*\|^2_{\mathcal M_{\mathcal X}} -\frac{1}{2} \| x_{k+1} - x_k\|_{\mathcal M_{\mathcal X}}^2  \\
      &+ \frac{\alpha}{2} \| x_{k+1} -x^* \|_{\mathcal B}^2 +  \frac{\alpha}{2}  \| x_{k+1} -x_k \|_{\mathcal B}^2 - \frac{\alpha}{2} \| x_{k} -x^* \|_{\mathcal B}^2 \\
    &+ \alpha D_f(u^*, u_{k+1}) + \alpha D_f( u_{k+1}, u_k) - \alpha D_f( u^*, u_k)\\
    &+ \alpha  D_{g_B}(p^*, p_{k+1}) +\alpha  D_{g_B}( p_{k+1}, p_k) -\alpha  D_{g_B}( p^*, p_k) .
\end{aligned}
\end{equation*}
Rewrite the inequality with $\mathcal E$ by rearranging the terms, we obtain
\begin{align*}
\mathcal E(x_{k+1}) - \mathcal E(x_{k}) \leq{}& - \frac{\mu \alpha}{2}\|x_{k+1} -x^*\|^2_{\mathcal M_{\mathcal X}} \\
& -\left [\frac{1}{2}\| x_{k+1}- x_k \|_{\mathcal M_{\mathcal X} - \alpha \mathcal B}^2 - \alpha D_f( u_{k+1}; u_k) - \alpha D_{g_B}( p_{k+1}; p_k) \right ]\\
\leq{}& - \frac{\mu \alpha}{2}\|x_{k+1} -x^*\|^2_{\mathcal M_{\mathcal X}}\\
\leq{}& - \frac{\mu \alpha}{2} \mathcal E(x_{k+1})
\end{align*}
where in the second inequality, by the proof of Lemma \ref{lem: positivity of Lyapunov}, the extra term is negative, and in the third equality, we use $\mathcal M_{\mathcal X}\geq \frac{1}{2}(\mathcal M_{\mathcal X} - \alpha \mathcal B)$ by a factorization similar to \eqref{eq:Mx-B}. 
\end{proof}


Theorem \ref{thm: convergence of G-S TPD} showed the step size is inversely proportional to the Lipschitz constants. Compared with the step size of the explicit schemes and IMEX schemes, which is also proportional to the convexity constants, the Lipschitz constants are usually easier to estimate.

\begin{remark}\rm 
If we further choose a large enough $\IQ$ (or scale appropriately) such that $L_S \leq 2$, then the upper bound of the step size can be enlarged to $\alpha < 1/\max \{ 4, 2L_{g_B, \IQ} \}$. For $\alpha = 1/ \max \left\{8 ,4L_{g_B,\IQ}  \right \}$, the convergence rate
\begin{equation*}
     \frac{1}{1+\mu \alpha/2} = \left (1+\frac{\min\left \{\mu_{\mV}, \mu_{\mQ}\right \}}{8\max \{L_{g_{B,\IQ}}, 2\}} \right)^{-1}.
\end{equation*}
In particular, when $g(p) = (b, p)$ is affine, $L_{g_B,\IQ} = L_S^2\leq 1$, we can choose constant step size $\alpha =1/8$ and get the linear rate
\begin{equation*}
    \frac{1}{1+\mu \alpha/2} = \frac{1}{1+\min\left \{\mu_{\mV}, \mu_{\mQ}\right \}/16}. 
\end{equation*}
\end{remark}

\section{Symmetric Transformed Primal-Dual Iterations}\label{sec:symmetric}

In this section, we present symmetric transformed primal-dual iterations which retain linear convergence when $f$ is strongly convex in the subspace $\ker(B)$ and may not be in the whole space.

\subsection{Symmetric transformed primal-dual flow}
To distinguish the role of transformation and preconditioners, we introduce SPD matrices $\IU, \IP$ for the transformation and treat $\IV$ and $\IQ$ as preconditioners. The change of variable associated with $\IU, \IP$ is given as
\begin{equation*}
\begin{aligned}
 v = u + \IU^{-1}B^{\top} p,  \quad q = p - \IP^{-1}Bu.
\end{aligned}
\end{equation*}
Recall that the strong convexity of the dual variable $p$ comes from the strong convexity of $g_B(p) = g(p) + \frac{1}{2} \left( B\IU^{-1}B^{\top}p, p\right)$. Symmetrically, define 
\begin{equation}\label{eq: fB}
    f_B(u) = f(u) + \frac{1}{2}(B^{\top}\IP^{-1}Bu, u ).
\end{equation}
With the spirit of transformation,  if $f_B(u)$ is strongly convex while $\mu_f = 0$, linear convergence rates can be still obtained by applying transformation to both the primal and dual variables. \correction{There are applications under this consideration, for example, see~\cite{chen2018multigrid} for solving Maxwell equations with divergence-free constraints.}

We present the symmetric transformed primal-dual (STPD) flow with $\IV, \IQ$ as \correction{preconditioners}:
\begin{equation}\label{eq: symmetric transformed primal-dual flow}
\left \{ \begin{aligned}
u' &=  \mathcal{G}^u(u,p)\\
p' &= \mathcal{G}^p(u,p)
\end{aligned} \right.
\end{equation}
with
\begin{equation}\label{eq: symmetric GuGp} 
    \begin{aligned}   
    \mathcal{G}^u(u,p) &= -\IV^{-1}(\partial_u \mathcal{L}(u,p)  + B^{\top}\IP^{-1}\partial_p \mathcal{L}(u,p)) \\
    &= -\IV^{-1}\left (\nabla f_B(u)+ B^{\top}(p - \IP^{-1} \nabla g(p)) \right ), \\
    \mathcal{G}^p(u,p) &= \IQ^{-1}\left ( \partial_p \mathcal{L}(u,p)-B\IU^{-1}\partial_u \mathcal{L}(u,p) \right ) \\
    &= - \IQ^{-1}\left(\nabla g_B(p) - B(u - \IU^{-1} \nabla f(u)) \right).
\end{aligned}
\end{equation}


The following lower bound of the cross terms can be proved like Lemma \ref{lem:key}. Here we state results with operators $\IU, \IP$.

\begin{lemma}\label{lem: key v}
Suppose $f \in \mathcal{S}_{\mu_{f,\IU}, L_{f,\IU}}$. For any $u_1,u_2\in \mV$ and $p_1, p_2\in \mQ$, we have
\begin{align*}
  & \langle \nabla f(u_1)-\nabla f(u_2), \IU^{-1}B^\top(p_1 - p_2) \rangle \\
    \geq{} &\frac{\mu_{f,\IU}}{2} \|v_1-v_2\|^2_{\IU} - \frac{L_{f,\IU}}{2}\| B^{\top}(p_1 - p_2)\|^2_{\IU^{-1}} - \frac{1}{2}\langle \nabla f(u_1)-\nabla f(u_2), u_1 - u_2 \rangle,
\end{align*}
where recall $v = u + \IU^{-1}B^{\top} p$.
\end{lemma}

\begin{lemma}\label{lem: key q}
Suppose $g \in \mathcal{S}_{\mu_{g,\IP}, L_{g,\IP}}$. For any $u_1,u_2\in \mV$ and $p_1, p_2\in \mQ$, we have
\begin{align*}
  & \langle \nabla g(p_1)-\nabla g(p_2), -\IP^{-1}B(u_1 - u_2) \rangle \\
    \geq{} &\frac{\mu_{g,\IP}}{2} \|q_1-q_2\|^2_{\IP} - \frac{L_{g,\IP}}{2}\| B(u_1 - u_2)\|^2_{\IP^{-1}} - \frac{1}{2}\langle \nabla g(p_1)-\nabla g(p_2), p_1 - p_2 \rangle,
\end{align*}
where recall $q = p - \IP^{-1}Bu$. In particular, when $g(p) = (b,p)$ is affine, the equality holds with all terms are 0.
\end{lemma}


 

\correction{The strong Lyapunov property and the Liptschitz continuity can be verified following the lines of proof in Section \ref{sec: transformed primal-dual flow}. For completeness, we present the results and skipped the proofs for brevity.}

\begin{theorem}\label{thm: strong Lyapunov property for symmetric TPD flow}
Choose $\IP$ such that $g(p) \in \mathcal{S}_{\mu_{g,\IP}, L_{g,\IP}} $ with $L_{g, \IP}\leq 1$. Choose $\IU$ such that $f(u) \in \mathcal{S}_{\mu_{f,\IU}, L_{f,\IU}}$ with $L_{f, \IU}\leq 1$ and assume $f_B$ is strongly convex, i.e, $\mu_{f_B, \IV} > 0$.  Then for the Lyapunov function~\eqref{eq: Lyapunov} and the STPD field $\mathcal G$~\eqref{eq: symmetric GuGp}, the following strong Lyapunov property holds
\begin{equation}\label{eq: strong lyapunov in for symmetric TPD flow}
    -\nabla \mathcal{E}(u,p) \cdot \mathcal{G}(u,p) \geq \mu \, \mathcal{E}(u,p) + \frac{\mu_{f,\IU}}{2}\| v - v^* \|_{\IU}^2 + \frac{\mu_{g,\IP}}{2} \|q-q^*\|^2_{\IP},
\end{equation}
where $0 < \mu = \min \left \{\mu_{f_B,\IV}, \mu_{g_B,\IQ}\right\}$. Consequently if $(u(t), p(t))$ solves the STPD flow~\eqref{eq: symmetric transformed primal-dual flow}, we have the exponential decay
$$\mathcal{E}(u(t),p(t)) \leq e^{-\mu t}\mathcal{E}(u(0),p(0)), \quad \forall t > 0.$$
\end{theorem}

\begin{remark}
\rm
The assumptions on Lipschitz constants can be relaxed to $L_{f, \IU} < 2$ and  $L_{g, \IP} < 2$, then the effective $\mu = \min \{\mu_{\mV}, \mu_{\mQ}\}$ is defined as
$$\mu_{\mV} = \min \{1, 2 - L_{f, \IU}\} \mu_{f_B, \IV}, \quad \mu_{\mQ} = \min \{1, 2 - L_{g, \IP}\} \mu_{g_B, \IQ}.$$
Therefore the algorithm is robust with perturbation on Lipschitz constants around $1$.  $\qed$
\end{remark}
To guarantee the exponential decay of the STPD flow, we require both $g_B$ and $f_B$ are strongly convex. In the linear saddle point system, this reduced to the necessary and sufficient conditions in~\cite{Zulehner2011} for the well-posedness of a saddle point problem. Especially for $g(p)=(b,p)$, it corresponds to the inf-sup condition for saddle point systems~\cite{brezziExistenceUniquenessApproximation1974}. 



Define 
\begin{equation}\label{eq: eUeP}
\begin{aligned}
     e_\mU = u - \IU^{-1} \nabla f(u),   \quad
     e_\mP = p - \IP^{-1} \nabla g(p)
\end{aligned}
\end{equation}
They are Lipschitz continuous as discussed in Section \ref{sec: eu} and the constants will be denoted by $L_{e_{\mU}, \IU}$ and $L_{e_{\mP}, \IP}$. 


%
%
%

\begin{lemma}\label{lem: Lipschitz continuity of sysmmetric TPD flow}
Assume $\nabla f_B$ and $\nabla g_B$ are Lipschitz continuous with Lipschitz constant $L_{f_B,\IV}$ and $L_{g_B, \IQ}$, respectively. Let $L_{e_{\mU},\IV}, L_{e_{\mP},\IQ}$ be the Lipschitz constant of $e_\mU, e_\mP$, respectively, then we have
\begin{equation*}
    \begin{aligned}
\|\mathcal{G}^u(u_1,p_1) - \mathcal{G}^u(u_2,p_2)\|_{\IV} &\leq L_{f_B, \IV}\|u_1-u_2\|_{\IV} +L_{e_{\mP}, \IQ} L_S \|p_1-p_2\|_{\IQ},\\
\|\mathcal{G}^p(u_1,p_1) - \mathcal{G}^p(u_2,p_2)\|_{\IQ} &\leq L_{g_B, \IQ}\|p_1-p_2\|_{\IQ} +L_{e_{\mU}, \IV} L_S \|u_1-u_2\|_{\IV},
    \end{aligned}
\end{equation*}
for all $u_1, u_2 \in \mV$ and $p_1, p_2 \in \mQ$.
\end{lemma}

\subsection{Explicit Euler method}
An explicit discretization for~\eqref{eq: symmetric transformed primal-dual flow} is as follows:
\begin{equation}\label{eq: symmetric EE method}
\left\{\begin{array}{l}\begin{aligned}
u_{k+1} &= u_k + \alpha_k \mathcal G^u(u_{k}, p_{k}), \\
p_{k+1}  &= p_k + \alpha_k \mathcal G^p(u_{k}, p_{k}).
\end{aligned}\end{array}\right.
\end{equation}
To compute the transformation, we introduce intermediate variables $u_{k+1/2}, p_{k+1/2}$ and present an equivalent but computationally favorable form of~\eqref{eq: symmetric EE method}:
\begin{equation}
\left\{\begin{array}{l}
   \begin{aligned}\label{eq: EE symmetric TPD}
u_{k+1/2} &= u_k - \IU^{-1}(\nabla f(u_k) + B^\top p_{k}),\\ 
p_{k+1/2} &= p_k - \IP^{-1}(\nabla g(p_k) - B u_{k}),\\
u_{k+1} &= u_k - \alpha_k \IV^{-1}\left(\nabla f(u_k) + B^{\top}p_{k+1/2} \right),\\
p_{k+1}  &= p_k - \alpha_k \IQ^{-1}\left(\nabla g(p_k) - B u_{k+1/2} \right).
\end{aligned}
\end{array}\right. 
\end{equation}
All four SPD operators can be scaled identities and scheme~\eqref{eq: EE symmetric TPD} can \correction{be interpreted} as two steps of primal-dual iterations with the same gradient $\nabla f(u_k)$ and $\nabla g(p_k)$. 
The convergence analysis is more clear in the formulation~\eqref{eq: symmetric EE method}. Follow the same proof of Theorem \ref{convergence rate for EE discretization}, we obtain the linear convergence of the scheme~\eqref{eq: EE symmetric TPD}.


\begin{theorem}\label{thm: convergence rate for symmetric EE method}
Choose $\IP$ such that $g(p) \in \mathcal{S}_{\mu_{g,\IP}, L_{g,\IP}} $ with $L_{g, \IP}\leq 1$ and choose $\IU$ such that $f(u) \in \mathcal{S}_{\mu_{f,\IU}, L_{f,\IU}}$ with $L_{f, \IU}\leq 1$. Assume $f_B$ is strongly convex, i.e, $\mu_{f_B, \IV} > 0$ and $g_B$ is strongly convex with $\mu_{g_B, \IQ}>0$. Let $(u_k, p_k)$ follows the explicit scheme~\eqref{eq: symmetric EE method} for the STPD flow with initial value $(u_0, p_0)$. For the Lyapunov function defined by~\eqref{eq: Lyapunov}, it holds that
$$\mathcal{E}(u_{k+1}, p_{k+1}) \leq (1-\delta_k)\mathcal{E}(u_k, p_k)$$
for $0 < \alpha_k < \displaystyle \min \left \{\mu_{f_B, \IV}/L_{\mV}^2, \mu_{g_B, \IQ}/L_{\mQ}^2 \right \}$
and $$0 < \delta_k = \min \left \{ \alpha_k(\mu_{f_B, \IV} - L_{\mV}^2\alpha_k), \alpha_k\left ( \mu_{g_B, \IQ}- L_{\mQ}^2\alpha_k\right) \right\}<1,$$ 
with $$L_{\mV}^2 = 2\left (L_{f_B, \IV}^2 +L_{e_{\mU}, \IV}^2 L_S^2 \right), \quad L_{\mQ}^2 = 2\left (L_{g_B, \IQ}^2 + L_{e_{\mP}, \IQ}^2 L_S^2\right).$$

\end{theorem}

Define 
$$\kappa_{\mV} = L_{\mV}/\mu_{f_B, \IV}, \quad \kappa_{\mQ} = L_{\mQ}/\mu_{g_B, \IQ}.$$
Theorem \ref{thm: convergence rate for symmetric EE method} shows the convergence rate is determined by $\kappa_{\mV} $ and $\kappa_{\mQ}$. 
For $f, g \in \mathcal{C}^2$, a guideline to choose $\IV, \IQ$ would be 
$$\IV \approx \nabla^2 f + B^{\top}\IP^{-1}B,\quad \quad \IQ \approx \nabla^2 g + B\IU^{-1}B^{\top}.$$
%
%
For affine $g(p) = (b, p)$, it is straightforward to show $L_{g,\IP} = 0$ and $L_{e_\mP, \IQ} = 1$ for any $\IP, \IQ$. Let $\IP = \IQ = I$, we can choose $\IU = \IV$ and $L_{f, \IU} \leq 1$ is satisfied by proper scaling. Then we have $\kappa_\mathcal Q = O (\kappa(B\IV^{-1}\correction{B^{\top}})).$
In this case, the convergence rate will be determined by $\kappa(B\IV^{-1}\correction{B^{\top}})$ and $\kappa_\mV$. The computational cost is basically the effort to compute $\IV^{-1}$. 

\subsection{Implicit-Explicit Methods} 
To get accelerated convergence rate, we can apply the IMEX scheme:
\begin{equation}\label{eq: symmetric IMEX method EpIu}
\left\{\begin{array}{l}\begin{aligned}
p_{k+1}  &= p_k + \alpha_k \mathcal G^p(u_k, p_k), \\
u_{k+1} &= u_k + \alpha_k  \mathcal G^u(u_{k+1}, p_{k+1}). 
\end{aligned}\end{array}\right.
\end{equation}
That is we update $p$ by the explicit Euler method and solve $u$ by the implicit Euler method. 
Again we can view~\eqref{eq: symmetric IMEX method EpIu} as a correction to the inexact Uzawa method
\begin{equation}\label{eq: symmetric IMEX algorithm}
\left\{\begin{array}{l}\begin{aligned}
u_{k+1/2} &= u_k - \IU^{-1}(\nabla f(u_k) + B^\top p_{k}),\\ 
p_{k+1}  &= p_k - \alpha_k \IQ^{-1}\left(\nabla g(p_k) - B u_{k+1/2} \right), \\
u_{k+1} &= \arg \min_{u \in \mV} \tilde f_B(u; u_k, p_{k+1}),
\end{aligned}\end{array}\right.
\end{equation}
where $$ \tilde f_B(u; u_k, p_{k+1}) = f_ B(u) + \frac{1}{2\alpha_k} \|u - u_k + \alpha_k\IV^{-1}B^\top\left(p_{k+1} -\IP^{-1} \nabla g(p_{k+1})\right)\|^2_{\IV}.$$
Compare with~\eqref{IMEX algorithm}, one more gradient descent step $p_{k+1} -\IP^{-1} \nabla g(p_{k+1})$ is added. 
When $\IV^{-1} = \frac{1}{L_f} I_m$, the last step is one proximal iteration
$$
u_{k+1} = {\rm prox}_{ f_B, \frac{\alpha_k}{L_f}} \left (u_k - \frac{\alpha_k}{L_f} B^\top \left(p_{k+1} -\IP^{-1} \nabla g(p_{k+1})\right) \right).
$$

The IMEX scheme enjoys accelerated linear convergence rates. We skipped the proof as it follows in line as Theorem \ref{convergence rate for IMEX discretization}.
\begin{theorem}\label{thm: convergence rate for symmetric IMEX method EpIu}
Choose $\IP$ such that $g(p) \in \mathcal{S}_{\mu_{g,\IP}, L_{g,\IP}} $ with $L_{g, \IP}\leq 1$ and choose $\IU$ such that $f(u) \in \mathcal{S}_{\mu_{f,\IU}, L_{f,\IU}}$ with $L_{f, \IU}\leq 1$. Assume $f_B$ is strongly convex, i.e, $\mu_{f_B, \IV} > 0$ and $g_B$ is strongly convex with $\mu_{g_B, \IQ}>0$. Let $(u_k, p_k)$ follows the IMEX scheme~\eqref{eq: symmetric IMEX algorithm} for the STPD flow with initial value $(u_0, p_0)$. For the Lyapunov function defined by~\eqref{eq: Lyapunov}, it holds that
$$\mathcal{E}(u_{k+1}, p_{k+1}) \leq \frac{1}{1+\alpha_k \mu_k}\mathcal{E}(u_k, p_k),$$
for $0 < \alpha_k < \displaystyle \mu_{g_B, \IQ}/L_{S,\mQ}^2$ and $\mu_k = \min \left \{\mu_{f_B, \IV} , \mu_{g_B, \IQ} - \alpha_k L_{S,\mQ}^2\right\}$, where $L_{S,\mQ}^2 =  L_{g_B, \IQ}^2 + L_{e_\mU, \IV}^2L_S^2$.
In particular, for $\alpha_k = 0.5 \mu_{g_B, \IQ}/ L_{S,\mQ}^2$, we have 
$$\mathcal{E}(u_{k+1}, p_{k+1}) \leq \frac{1}{1+ 0.5 \mu_{g_B, \IQ} \min \{\mu_{f_B, \IV}, \mu_{g_B, \IQ}/2\} / L_{S,\mQ}^2 }\mathcal{E}(u_k, p_k).$$
\end{theorem}

The inner solve in~\eqref{eq: symmetric IMEX algorithm} can be relaxed to an inexact solver. We state the result as a corollary of Theorem \ref{thm: convergence of inexact IMEX TPD}.

\begin{corollary}
Choose $\IP$ such that $g(p) \in \mathcal{S}_{\mu_{g,\IP}, L_{g,\IP}} $ with $L_{g, \IP}\leq 1$ and choose $\IU$ such that $f(u) \in \mathcal{S}_{\mu_{f,\IU}, L_{f,\IU}}$ with $L_{f, \IU}\leq 1$. Assume $f_B$ is strongly convex, i.e, $\mu_{f_B, \IV} > 0$ and $g_B$ is strongly convex with $\mu_{g_B, \IQ}>0$. Suppose $(u_k, p_k)$ follows the inexact IMEX iteration~\eqref{eq: symmetric IMEX algorithm} with initial value $(u_0, p_0)$ and the inexact inner solver returns $u_{k+1}$
satisfying $\|\nabla \tilde f_B(u_{k+1})\|^2_{\IV^{-1}} \leq \epsilon_k$ for $k = 1,2, \cdots$. Then for the Lyapunov function defined by~\eqref{eq: Lyapunov}, it holds that
$$\mathcal{E}(u_{k+1}, p_{k+1}) \leq \frac{1}{1+\alpha_k \mu_k}\mathcal{E}(u_k, p_k)+ \frac{\alpha_k}{(1+\alpha_k \mu_k)\mu_{\mV}}\epsilon_k ,$$
for $0 < \alpha_k < \displaystyle \mu_{g_B, \IQ}/L_{S,\mQ}^2$ and $\mu_k = \min \left \{\mu_{f_B, \IV}/2, \mu_{g_B, \IQ} - \alpha_k L_{S,\mQ}^2 \right \}$, where $L_{S,\mQ}^2 =  L_{g_B, \IQ}^2 + L_{e_\mU, \IV}^2L_S^2 $.
In particular, for $\alpha_k = \mu_{g_B, \IQ}/ 2L_{S,\mQ}^2$, the accumulative perturbation error for the inexact solve is
$$\mathcal{E}(u_{n+1}, p_{n+1}) \leq \rho^{n+1}\mathcal{E}(u_0, p_0)+ \frac{\mu_{g_B, \IQ}}{2\mu_{f_B, \IV}L_{S,\mQ}^2}\sum_{k=0}^n \rho^{n-k+1} \epsilon_k ,$$
where $\mu = \min\{\mu_{f_B, \IV}, \mu_{g_B, \IQ}\}$ and $\rho =1/(1+\mu_{g_B, \IQ} \mu/ 4L_{S,\mQ}^2)\in (0,1)$.
\end{corollary}

\correction{Due to the nonlinear coupling $B^{\top}(p - \IP^{-1} \nabla g(p))$, we cannot apply GS-AOR scheme to STPD in general. Only when $g$ is affine, i.e., the constrained optimization problems, $\nabla g$ is constant, the Gauss-Seidel splitting can be adapted to STPD and achieve the accelerated linear convergence. For this case, it can be also retrieved by considering augmented Lagrangian and apply TPD. We shall discuss this important case in the following section.} 


\section{Augmented Lagrangian Methods}\label{sec:ALM}
In this section, we consider the augmented Lagrangian methods~\cite{hestenes1969multiplier, powell1969method} for solving the constrained optimization problem~\eqref{eq: one-block affine equality constrained optimization system}. 
Consider the augmented Lagrangian
\begin{equation}\label{eq: augmented Lagrangian problem}
   \min_{u \in \mathbb{R}^m} \max_{p \in  \mathbb{R}^n} \mathcal{L}_\beta (u,p) = f(u) + \frac{\beta}{2} \|Bu-b\|^2 + (p, Bu-b),
\end{equation}
where $\beta \geq 0$. It is clear that the critical points of $\mathcal{L}_\beta (u,p)$ are equivalent for all $\beta$, as the constraint $Bu = b$ holds for critical points, and when $\beta = 0$,~\eqref{eq: augmented Lagrangian problem} returns to the Lagrangian of the constrained optimization problem \eqref{eq: one-block affine equality constrained optimization system}. 

Notice~\eqref{eq: augmented Lagrangian problem} is still a nonlinear saddle point system with $g(p) = (b,p)$ and $f_\beta(u) = f(u) + \frac{\beta}{2}\|Bu-b\|^2$, the TPD flow and the corresponding transformed primal-dual iterations can be adapted. In this section, we will show that simple choices of $\IQ = \beta I_n$ in the TPD flow is a good preconditioner for solving augmented Lagrangian when $\beta$ is sufficiently large. Particular discrete schemes will recover a class of augmented Lagrangian methods. 

ALM can be also derived from STPD flow for the original Lagrangian by using $\IP = \beta I$ and thus enhance the stability by the strong convexity of $f_B$. 
We first show the strong convexity equivalence between a simplified $f_B$ and $f_\beta$, where
$$f_B(u) =  f(u) + \frac{1}{2} (B^{\top}Bu, u), \quad f_\beta(u) = f(u) + \frac{\beta}{2}\|Bu-b\|^2.$$

\begin{lemma}
For any $\beta > 0$, $f_B$ is strongly convex if and only if $f_\beta$ is strongly convex. In particular, $\mu_{f_\beta} \geq \mu_{f_B}$ for $\beta \geq 1$.
\end{lemma}

\begin{proof}

Suppose $f_B$ is $\mu_{f_B}$-strongly convex with $\mu_{f_B} > 0$, for all $u_1, u_2 \in \mathcal{V}$,
\begin{equation*}
\begin{aligned}
\langle \nabla f_\beta(u_1) - \nabla f_\beta (u_2), u_1 - u_2 \rangle \geq{}& \min\{\beta, 1\}\langle \nabla f_B(u_1) - \nabla f_B (u_2), u_1 - u_2 \rangle \\
\geq& \min\{\beta, 1\}\mu_{f_B}\|u_1-u_2\|^2.
\end{aligned}.
\end{equation*}
Hence $f_\beta$ is $\mu_{f_\beta}$-strongly convex with $\mu_{f_\beta} \geq \min\{\beta, 1\}\mu_{f_B} > 0$. For $\beta \geq 1, \mu_{f_\beta} \geq \mu_{f_B}$.

Suppose $f_\beta$ is $\mu_{f_\beta}$-strongly convex with $\mu_{f_\beta} > 0$, for all $u_1, u_2 \in \mathcal{V}$,
\begin{equation*}
\begin{aligned}
\langle \nabla f_B(u_1) - \nabla f_B (u_2), u_1 - u_2 \rangle \geq{}& \min\{\beta^{-1}, 1\}\langle \nabla f_\beta(u_1) - \nabla f_\beta (u_2), u_1 - u_2 \rangle \\
\geq&  \min\{\beta^{-1}, 1\}\mu_{f_B}\|u_1-u_2\|^2.
\end{aligned}.
\end{equation*}
Hence $f_B$ is $\mu_{f_B}$-strongly convex with $\mu_{f_B} = \min\{\beta^{-1}, 1\}\mu_{f_\beta} > 0$.
\end{proof}

%
Therefore ALM can achieve linear convergence rate even $f$ is not strongly convex but $f_B$ is. Besides the enhanced stability, next we shall interpret the augmented Lagrangian as a preconditioner of the Schur complement: for sufficiently large $\beta$, a simple choice $\IQ^{-1}=\beta I$ will lead to a well conditioned $\kappa_{\mQ}$. The condition number $\kappa_{\mV}$ will be controlled by using another SPD matrix $A$. 

%

\begin{proposition}\label{preconditioned Schur complement}
Let $A$ be an SPD matrix and define $A_{\beta} = A+\beta B^\top B$ for $\beta > 0$. Assume $f_B(u) \in \mathcal{S}_{\mu_{f_B,A_1}, L_{f_B,A_1}}$. Choose 
$$
\IV^{-1} = A_{\beta}^{-1} = \left( A+\beta B^\top B \right )^{-1}, \quad \IQ ^{-1} = \beta I_n.
$$
Then for $\beta \geq 1$
\begin{equation}\label{eq:fbetaLmu}
\min\{\mu_{f_B,A_1}, 1\}\leq \mu_{f_{\beta},\IV} \leq L_{f_{\beta},\IV} \leq \max \{L_{f_B,A_1}, 1\},
\end{equation}
and
\begin{equation}\label{eq:Sbeta}
\frac{\mu_{S_0}}{1+ \beta \mu_{S_0}} \leq \lambda_{\min} \left (BA_{\beta}^{-1}B^\top \right) \leq  \lambda_{\max} \left (BA_{\beta}^{-1}B^\top \right) \leq \frac{1}{\beta},
\end{equation}
where $\mu_{S_0} = \lambda_{\min} (BA^{-1}B^{\top})$. 
Consequently 
$$
\kappa_{\IV}(f_{\beta}) \leq \kappa_{A_1}(f_B), \quad \kappa( \IQ^{-1}B\IV^{-1}B^\top) \leq 1 + \frac{1}{\beta \mu_{S_0}}.
$$
\end{proposition}
\begin{proof}
Bound~\eqref{eq:fbetaLmu} is straight forward. Define $S_{\beta} = B \left( A+\beta B^\top B \right )^{-1}B^{\top}$. 
By Woodbury matrix identity,
\begin{equation*}
    \begin{aligned}
    B A_{\beta}^{-1} B^\top &=B \left ( A+\beta B^\top B \right)^{-1} B^\top \\
    &= B\left (A^{-1} -A^{-1}B^\top \left(\beta^{-1}I_n + BA^{-1}B^\top \right) ^{-1}BA^{-1} \right) B^\top  \\
    &= S_0 -S_0 \left(\beta^{-1}I_n + S_0 \right) ^{-1}S_0.
    \end{aligned}
\end{equation*}
Hence
\begin{equation*}
    \begin{aligned}
    \sigma \left(B A_{\beta}^{-1} B^\top \right) &= \sigma(S_{\beta}) =  \left \{\frac{\lambda}{ 1+ \beta \lambda} , \lambda \in \sigma(S_0)\right \}.
    \end{aligned}
\end{equation*}
Then~\eqref{eq:Sbeta} follows. 
\end{proof}

As an example, if we choose $\beta \geq 1/\mu_{S_0}$, then the condition number of the Schur complement is bounded by $2$. While the condition number of $f_{\beta}$ keeps unchanged and preconditioning of $f$ can be achieved by appropriate choice of $A$. The condition number for the primary variable is bounded by $\kappa_{A_1}(f_B)$. 

In practice, $\left( A+\beta B^\top B \right )^{-1}$ can be further relaxed to an inexact solver $\IV^{-1}$ which introduce a factor $\lambda_{\min} (\IV^{-1}A_{\beta})$ in the convergence rate. In the sequel, we shall fix the simple choice $\IQ^{-1} = \beta I_n$ and $\beta \gg 1$. We can either apply discretization \correction{of the TPD flow} to the augmented Lagrangian ~\eqref{eq: augmented Lagrangian problem} or \correction{the STPD flow} to the original Lagrangian \correction{$\mathcal{L} (u,p) = f(u) - (b,p) + (Bu, p)$}. The resulting schemes are slightly different but share similar convergence rate. Here is an example. 

The explicit scheme of the TPD flow for the augmented Lagrangian \correction{(ALM-Explicit)} is:
\begin{equation}\label{EE discretization for ALM}
\left\{\begin{array}{l}\begin{aligned}
u_{k+1/2} &= u_k - \IV^{-1}\left (\nabla f(u_k) + \beta B^{\top}(Bu_k-b) + B^\top p_{k}\right ),\\ 
p_{k+1}  &= p_k - \alpha_k \beta \left(b - B u_{k+1/2} \right),\\
u_{k+1} &= u_k - \alpha_k \IV^{-1} \left (\nabla f(u_k) + \beta B^{\top}(Bu_k-b) + B^\top p_{k}\right ).
\end{aligned}\end{array}\right.
\end{equation}
Computationally the third step can be written as $u_{k+1} = (1-\alpha_k)u_k + \alpha_k u_{k+1/2}.$
The explicit scheme of \correction{the STPD flow for the Lagrangian} with $\IP^{-1} = \IQ^{-1} = \beta I$:
\begin{equation}\label{eq: symmetric EE TPD for AL}
\left\{\begin{array}{l}\begin{aligned}
u_{k+1/2} &= u_k - \IU^{-1}(\nabla f(u_k) + B^\top p_k), \\
p_{k+1}  &= p_k - \alpha_k \beta \left(b - B u_{k+1/2} \right),\\
u_{k+1} &= u_k - \alpha_k \IV^{-1} \left (\nabla f(u_k) + \beta B^{\top}(Bu_k-b) + B^\top p_{k}\right ).
\end{aligned}\end{array}\right.
\end{equation}
So~\eqref{EE discretization for ALM} and~\eqref{eq: symmetric EE TPD for AL} are only different in the first step of updating $u_{k+1/2}$:~\eqref{eq: symmetric EE TPD for AL} is the gradient flow of $u$ using $\partial_u \mathcal L$, and~\eqref{EE discretization for ALM} is $\partial_u \mathcal L_{\beta}$. 
%
\correction{Discretization} of the TPD or STPD flow gives generalized variants of augmented Lagrangian-like methods and provide flexibility of choosing transformation operators and preconditioners. Within our framework, one can easily derive convergence analysis by verification of assumptions.

Next we present the convergence analysis. To save space, we only present the version of TPD flow for $\mathcal L_{\beta}$. The STPD flow for $\mathcal{L}$ is similar.

\begin{theorem}\label{thm: convergence rate for EE discretization for augmented Lagrangian}
Let $A$ be an SPD matrix and define $A_{\beta} = A+\beta B^\top B$ for $\beta > 0$. Assume $f_B(u) \in \mathcal{S}_{\mu_{f_B,A_1}, L_{f_B,A_1}}$ with $0 < \mu_{f_B,A_1} \leq L_{f_B,A_1} \leq 1$. Choose $\IV^{-1}$ such that $\lambda_{\max}(\IV^{-1}A_{\beta}) \leq 1.$
Let $(u_k, p_k)$ follows iteration~\eqref{EE discretization for ALM} with initial value $(u_0, p_0)$, it holds that
$$\mathcal{E}(u_{k+1}, p_{k+1}) \leq (1-\delta_k)\mathcal{E}(u_k, p_k)$$
for $0 < \alpha_k < \mu/4$ with $\mu := \displaystyle \min \left \{\mu_{\mV }, \mu_{\mQ }\right \}$
and $$\delta_k = \min \left \{ \alpha_k(\mu_{\mV} - 4\alpha_k), \alpha_k\left (\mu_{\mQ} - 4\alpha_k\right) \right\},$$
where 
$$
\mu_{\mV} = \mu_{f_B,A_1}\lambda_{\min} (\IV^{-1}A_{\beta}), \quad \mu_{\mQ} = \frac{\beta \mu_{S_0}}{1+ \beta \mu_{S_0}}\lambda_{\min} (\IV^{-1}A_{\beta})
$$
with $\mu_{S_0} =\lambda_{\min} (BA^{-1}B^{\top})$. 

In particular for $\alpha_k = \mu/8$, we have
$$\mathcal{E}(u_{k+1}, p_{k+1}) \leq \left (1 - \frac{\mu^2}{16}\right )\mathcal{E}(u_k, p_k).$$
\end{theorem}
\begin{proof}
%

By~\eqref{eq:fbetaLmu} and assumption $L_{f_B,A_1} \leq 1$, we have $L_{f_{\beta},\IV} \leq 1$. Consequently we can apply Theorem \ref{convergence rate for EE discretization}. 

To estimate the constants, we introduce a partial ordering for symmetric matrices. For two symmetric matrices $X, Y$, we say $X \preceq Y$ if $Y - X$ is positive semidefinite. Then
\begin{equation}\label{eq: IV Abeta condition number}
\lambda_{\min} (\IV^{-1} A_{\beta})\IV \preceq A_{\beta} \preceq \lambda_{\max} (\IV^{-1} A_{\beta})\IV,
\end{equation}
\begin{equation}\label{IV Abeta Schur condition number}
\lambda_{\min} (\IV^{-1}A_{\beta})BA_{\beta}^{-1} B^\top \preceq B\IV^{-1}B^\top \preceq \lambda_{\max} (\IV^{-1}A_{\beta})BA_{\beta}^{-1}  B^\top.
\end{equation}

By Proposition \ref{preconditioned Schur complement} and~\eqref{IV Abeta Schur condition number}, since $\lambda_{\max}(\IV^{-1}A_{\beta}) \leq 1$,
\begin{equation*}
    \begin{aligned}
    L_{g_B, \IQ} = L_S^2 &=  \lambda_{\max}(\IQ^{-1}B\IV^{-1}B^\top) = \beta \lambda_{\max} (B\IV^{-1}B^\top) \\
    & \leq \beta \lambda_{\max} (\IV^{-1}A_{\beta})\lambda_{\max} \left (BA_{\beta}^{-1}B^\top \right )  \leq1. 
    \end{aligned}
\end{equation*}
Therefore,
\begin{equation*}
    \begin{aligned}
L_{\mV}^2 &= 2\left ( L_{e_{\beta}, \IV}^2(1+L_S^2) \right) \leq 4, \\
L_{\mQ}^2 &= 2\left ( L_{g_B, \IQ}^2 +L_S^2 \right) \leq 4,\\
     \end{aligned}
\end{equation*}
\correction{where $e_\beta (u)= u - \IV^{-1} \nabla f_\beta(u)$.}

Similarly,
\begin{equation*}
    \begin{aligned} \mu_{g_B, \IQ} & = \lambda_{\min}(\IQ^{-1}B\IV^{-1}B^\top) = \beta \lambda_{\min} (B\IV^{-1}B^\top)  \\
    & \geq \beta \lambda_{\min} (\IV^{-1}A_{\beta})\lambda_{\min} \left (BA_{\beta}^{-1}B^\top \right )  \geq \lambda_{\min} (\IV^{-1}A_{\beta}) \frac{\beta\mu_{S_0}}{1+ \beta \mu_{S_0}}.    \end{aligned}
\end{equation*}
Thus we have 
\begin{equation*}
    \begin{aligned}
    \mu_{\mV} =  \mu_{f_B,A_1}\lambda_{\min} (\IV^{-1}A_{\beta}), \quad \mu_{\mQ} = \frac{\beta \mu_{S_0}}{1+ \beta \mu_{S_0}}\lambda_{\min} (\IV^{-1}A_{\beta}),
    \end{aligned}
\end{equation*}
and desired estimate then follows.
\end{proof}

The assumption $L_{f,A}\leq 1$ and $\lambda_{\max}(\IV^{-1}A_{\beta}) \leq 1$ can be easily satisfied by scaling. For example, if $L_{f,A} > 1$, we can assign $L_{f,A}A$ as a new $A$. Once $A_{\beta}$ is available, symmetric Gauss-Seidel or V-cycle multigrid iteration will define an $\IV^{-1}$ with $\lambda_{\max}(\IV^{-1}A_{\beta}) \leq 1$. As the upper bound requirement is $L_{f_\beta, \IV}< 2$, the analysis and algorithm is robust to small perturbation near $L_{f_\beta, \IV} = 1$. 

\correction{
In the following we present the GS-AOR for the augmented Lagrangian~\eqref{eq: augmented Lagrangian problem} (ALM-GS-AOR): 
\begin{equation}\label{eq:ALM-GS TPD}
\left \{\begin{aligned}
  \frac{u_{k+1} - u_k}{\alpha}= & -\IV^{-1}(\nabla f(u_{k}) + \beta B^{\top}(Bu_k-b) + B^{\top}p_{k}) \\
\frac{p_{k+1} - p_k}{\alpha} = &-\beta \left [ B\IV^{-1}B^{\top}p_k +b- B(2u_{k+1} - u_k) \right.\\
&\left. + B\IV^{-1}\left (\nabla f(u_{k+1}) + \beta B^{\top}(Bu_{k+1}-b)\right ) \right ].
\end{aligned} \right.
\end{equation}
\begin{theorem}\label{thm: convergence rate for GS discretization for augmented Lagrangian}
Let $A$ be an SPD matrix and define $A_{\beta} = A+\beta B^\top B$ for $\beta > 0$. Assume $f_B(u) \in \mathcal{S}_{\mu_{f_B,A_1}, L_{f_B,A_1}}$ with $0 < \mu_{f_B,A_1} \leq L_{f_B,A_1} \leq 1$. Choose $\IV^{-1}$ such that $\lambda_{\max}(\IV^{-1}A_{\beta}) \leq 1.$
Let $(u_k, p_k)$ follows iteration~\eqref{eq:ALM-GS TPD} with initial value $(u_0, p_0)$, it holds that
$$
\mathcal{E}(u_{k+1}, p_{k+1}) \leq  \frac{1}{ 1 +\mu \alpha/2}\mathcal{E}(u_k, p_k)$$
for $0 < \alpha < 1/4$ with $\mu := \displaystyle \min \left \{\mu_{\mV }, \mu_{\mQ }\right \}$
where 
$$
\mu_{\mV} = \mu_{f_B,A_1}\lambda_{\min} (\IV^{-1}A_{\beta}), \quad \mu_{\mQ} = \lambda_{\min} (\IV^{-1}A_{\beta})\frac{\beta \mu_{S_0}}{1+ \beta \mu_{S_0}}
$$
with $\mu_{S_0} =\lambda_{\min} (BA^{-1}B^{\top})$.
In particular for $\alpha = 1/8$, we have
$$\mathcal{E}(u_{k+1}, p_{k+1}) \leq \frac{1}{ 1 + \mu/16} \mathcal{E}(u_k, p_k).$$
\end{theorem}
\begin{proof}
By~\eqref{eq:fbetaLmu} and assumption $L_{f_B,A_1} \leq 1$, we have $L_{f_{\beta},\IV} \leq 1$. Consequently we can apply Theorem \ref{thm: convergence of G-S TPD}. The desired result follows from the constant bounds given in Theorem \ref{thm: convergence rate for EE discretization for augmented Lagrangian}.
\end{proof}
}

In Table \ref{table:ALM examples},  we list out typical choices of $\IV^{-1}$ and compare TPD and ALM schemes for convex optimization problems with affine equality constraints~\eqref{eq: one-block affine equality constrained optimization system}. Explicit schemes only require linear SPD solvers, but the convergence rate is $O(1-1/\kappa^2(f))$ or $O(1-1/\kappa_A^2(f))$. If the proximal operator of $f$ is available and $(BB^{\top})^{-1}$ can be efficiently computed, we can apply the IMEX 1 to accelerate converge rate to $ O(1-1/\kappa(f))$. If some preconditioner $A^{-1}$ of $f$ is given, then the convergence rate can be accelerated to $ O(1-1/\kappa_A(f))$ using TPD-IMEX 2 scheme. However, an inner solver to a nonlinear strongly convex optimization problem is required. \correction{Overall we recommend the GS-AOR methods, which enjoy a convergence rate of $(1+c/\kappa)^{-1}$ and only require linear SPD solvers. When $f$ is not strongly convex, we recommend to use ALM-GS-AOR which can enhance the convexity to $f_B$.}

\begin{table}
	\centering
	\caption{Examples of $\IV^{-1}$ and $\IQ^{-1}$ for $f\in \mathcal S_{\mu_f, L_f}$ or $f\in \mathcal S_{\mu_{f,A}, L_{f,A}}$ and $g(p) = (b,p)$. $A$ is an SPD matrix induced inner product in $\mV$ with $L_{f,A}\leq 1$.} 
	\renewcommand{\arraystretch}{1.65}
	\begin{tabular}{@{} c c c  c @{}}
	\toprule
		&  \multicolumn{2}{c}{Linear inner solvers}& Rate\\
		\cline{2-3}	
		&   $\IV^{-1}$ & $\IQ^{-1}$   & $\beta \gg 1$		
 \\  \hline
 Explicit 1 & $\frac{1}{L_f}I_m$ & $L_f (BB^\top)^{-1}$ & $1-1/\kappa^{2} (f)$

\medskip \\


Explicit 2 & $A^{-1}$ & $(BA^{-1}B^{\top})^{-1}$ & $ 1- 1/\kappa^{2}_A(f)$

\medskip \\

IMEX 1 & $\frac{1}{L_f}I_m$ & $L_f (BB^\top)^{-1}$  & $ \left(1+ 1/\kappa(f)\right)^{-1}$

\medskip \\

&nonlinear solver & \multicolumn{2}{c}{${\rm prox}_{ f, \frac{\alpha_k}{L_f}}  (u_k - \frac{\alpha_k}{L_f} B^\top p_{k+1})$}

\medskip \\
IMEX 2 & $A^{-1}$ & $(BA^{-1}B^{\top})^{-1}$  & $ \left(1+ 1/\kappa_A(f)\right)^{-1}$

\medskip \\

&\multicolumn{3}{c}{nonlinear solver \quad $\min_{u \in \mV} f(u) + \frac{1}{2\alpha_k} \|u - u_k + \alpha_k\IV^{-1}B^\top p_{k+1}\|^2_{A}$}

\medskip \\

\correction{GS-AOR 1} & $\frac{1}{L_f}I_m$ & $L_f (BB^\top)^{-1}$  & $ \correction{\left(1+ 1/\kappa(f)\right)^{-1}}$

\medskip \\

\correction{GS-AOR 2} & $A^{-1}$ & $ (BA^{-1}B^\top)^{-1}$  & $ \correction{\left(1+ 1/\kappa_A(f)\right)^{-1}}$

\medskip \\
 
ALM-Explicit 1 & $(L_f I_m + \beta B^{\top}B)^{-1}$ & $\beta I_n$ &  $1- 1/\kappa^2(f)$
\medskip \\

ALM-Explicit 2 & $(A + \beta B^{\top}B)^{-1}$ & $\beta I_n$ & $1- 1/\kappa^2_A(f)$
\medskip \\




\correction{ALM-GS-AOR 1} & $(L_f I_m + \beta B^{\top}B)^{-1}$ & $\beta I_n$ &  $ \correction{\left(1+ 1/\kappa(f_B)\right)^{-1}}$

\medskip \\

\correction{ALM-GS-AOR 2} & $(A + \beta B^{\top}B)^{-1}$ & $\beta I_n$ & $ \correction{\left(1+ 1/\kappa_A(f_B)\right)^{-1}}$
\medskip \\
\bottomrule
	\end{tabular}
	\label{table:ALM examples}
\end{table}

Our analysis on ALM shows that the condition number of $f$ and Schur \correction{complement} can be simultaneously improved with a modified linear solver $(A+\beta B^\top B)^{-1}$ or a modified inner problem for $f_{\beta}$. Compared with  schemes without ALM, update of the dual variable in ALM is simpler and more importantly the stability is enhanced from the symmetrized transformed primal-dual flow point of view. 

\section{Conclusion and Future Work}

By revealing `Schur complement' in the transformed primal-dual flow, we proposed first-order algorithms, the Transformed Primal-Dual (TPD) iterations, and achieve linear convergence rates without the strong convexity of function $f$ or $g$. From a perspective of change of variables, the convergence rate in our analysis is essentially determined by choices of inner products on the primal and dual spaces. The augmented Lagrangian methods can enhance the stability and preconditioning the Schur complement so that the scaled identity defines a suitable inner product in the dual space. We also derive an approach to analyze the inexact inner solvers with perturbation on the gradient norm of a modified objective function for the sub-problem. \correction{More importantly, we propose a Gauss-Seidel iteration with accelerated overrelaxation (GS-AOR) to the TPD flow to obtain accelerated linear rate $(1+c/\kappa)^{-1}$.}

\correction{For the strongly-convex-strongly-concave nonlinear saddle point system, the optimal lower bound rate $(1+c/\sqrt{\kappa})^{-1}$ for first-order methods  is recently proved in~\cite{zhang2022lower}. We shall develop accelerated primal-dual methods to reach this rate and extend to convex-concave saddle point problems by combing the TPD flow.} 

Multigrid methods have been developed for linear saddle point systems ~\cite{bacuta2006unified,chen2018multigrid} and convex optimization problems~\cite{chen2020convergence}, showing convergence independent of problem sizes. One of our future work will be deriving multigrid-like methods for nonlinear saddle point systems. The TPD iterations can be used as good smoothers. Furthermore, we will extend this framework to tackle more general nonlinear saddle point systems, such as non-smooth objective function $f$, variables $(u,p)$ restricted in convex sets. For multi-block problems, the TPD flow will connect to the alternating direction method of multipliers (ADMM) ~\cite{boyd2011distributed,deng2016global} and there relation deserves further investigation.

\bibliographystyle{abbrv}
\bibliography{references,SaddleProblem}

\end{document}